\documentclass[11pt]{article}

\usepackage{tasks}

\usepackage{amssymb,amsmath}
\usepackage{amsthm}
\usepackage[top=0.5in,bottom=1in,left=1.25in,right=1.25in]{geometry}
\usepackage[colorlinks=true, linkcolor=blue,filecolor=blue, urlcolor=blue, citecolor=cyan,pagebackref=true]{hyperref}

\usepackage[all]{xy}
\usepackage{graphicx}
\usepackage{tikz-cd}  
\usetikzlibrary{cd,matrix, positioning, backgrounds, shapes.arrows, arrows, patterns, calc}	
\usetikzlibrary{positioning,arrows,calc}
\usepackage{mathabx}
\usepackage{bm}
\usepackage{wasysym}
\usepackage{empheq}
\usepackage{float}
\usepackage{booktabs}
\usepackage{tabularx}
\usepackage{enumitem}

\newtheorem{theorem}{Theorem}[section]
\newtheorem{definition}[theorem]{Definition}

\newtheorem{lemma}[theorem]{Lemma}

\newtheorem{proposition}[theorem]{Proposition}
\newtheorem{cor}[theorem]{Corollary}
\newtheorem{example}[theorem]{Example}
\newtheorem{remark}[theorem]{Remark}
\newtheorem{notation}[theorem]{Notation}

\newtheorem{theoremintr}{Theorem}

\numberwithin{equation}{section}

\newcommand{\Gal}{{\rm Gal}}
\newcommand{\Z}{\mathbb Z}
\newcommand{\N}{\mathbb N}
\newcommand{\Q}{\mathbb Q}
\newcommand{\eps}{\varepsilon}
\newcommand{\ev}{\operatorname{ev}}

\newcommand{\sym}{\operatorname{Sym}}
\newcommand{\R}{\mathbb{R}}
\newcommand{\zN}{\zeta_N}

\title{On structured cosine sums and applications}
\author{Qin Xue\footnote{Email: {\tt xueqin@westlake.edu.cn}}}
\date{\today}
\begin{document}

\maketitle
\begin{abstract}
		For a multiset $S$ on the cyclic group $\Z/N\Z$, we study finite sums of cosine functions of rational angles associated to $S$ by translating them as evaluations of elements in the group ring $Z[\Z/N\Z]$. Using vanishing sums of roots of unity,
especially the Lam--Leung theory, we obtain  criteria for the vanishing of
 the cosine sums under some conditions, and prove a  small-weight Fourier rigidity.  We then
apply these algebraic results to cyclic Cayley graphs, deriving the zero-eigenvalue
criteria, multiplicity bounds for nonzero eigenvalues in small-support case, and a description of the
square-free case where the generating set is a subgroup of the unit group.
		
		\vspace{4mm}

		\noindent{\it Keywords}: Vanishing sum,  Group ring, Cayley graph, Eigenvalue,
Multiplicity, Fourier rigidity. 
	\end{abstract}
\tableofcontents

\section{Introduction}

Trigonometric Diophantine equations with rational angles can be reduced,
after elementary trigonometric identities and clearing denominators, to finite
sums of cosines of rational multiples of $2\pi$.  Such rational
trigonometric equations were studied by Conway-Jones
\cite{conway1976trigonometric}.  Motivated by this perspective, we study the
following structured  cosine sums: for an integer $N\geq 1$ and a
multiset $S=\{s_1,\ldots,s_t\}\subseteq \Z/N\Z$, define
\begin{equation}\label{eq:intro-CS}
   C_S(k)=\sum_{j=1}^t
   \cos\!\left(\frac{2\pi k s_j}{N}\right),
   \qquad k\in \Z/N\Z .
\end{equation}
We are mainly interested in the following two questions for a fixed multiset $S$:
\begin{itemize}
    \item \emph{Vanishing problem:} characterize $k$ and $S$ for which $C_S(k)=0$?
    \item \emph{Multiplicity problem:} for a fixed $c\in \R$, estimate and determine the size of $\{k\in \Z/N\Z:C_S(k)=c\}$.
\end{itemize}
 
We aim to investigate these questions by translating \eqref{eq:intro-CS} into the language of roots of
unity and integral group rings.  Let $G:=\Z/N\Z$ and write $\Z[G]$ for the
integral group ring with basis elements $[x]$, $x\in G$.  The natural evaluation map $\ev_N:\Z[G]\to \mathbb{C}$ is given by
\begin{equation*}
  \ev_N\!\left(\sum_{x\in G} c(x)[x]\right)=
  \sum_{x\in G} c(x)\zeta_N^x .
\end{equation*}
For a multiset $X$ on $G$ we always identify $X$ with its incidence element
$\sigma(X):=\sum\limits_{x\in X}[x]\in \N[G]$.  A multiset $X$ on
$G$ can also be identified with its multiplicity function $X:G\to \N$. Its weight is $|X|:=\sum\limits_{g\in G}X(g)$. We say that $X$ is symmetric if $X(g)=X(-g)$ for all $g\in G$. We also put $X^*=-X$ and
$\sym(X)=X+X^*$.  Then
\begin{equation}\label{eq:intro-bridge}
  2C_S(k)=\ev_N\bigl(\sym(kS)\bigr).
\end{equation}
Thus the
vanishing and multiplicity questions of $C_S(k)$ are transformed to questions about vanishing
sums of $N$-th roots of unity, i.e., elements in $\ker(\ev_N) \cap \mathbb{N}[G]$.

Vanishing sums of roots of unity give a natural algebraic framework for
trigonometric Diophantine equations with rational angles. Indeed, the
identities $2\cos\theta=e^{i\theta}+e^{-i\theta}$ and $2i\sin\theta=e^{i\theta}-e^{-i\theta}$ transform relations among sines and cosines of rational multiples of $\pi$
into equations in roots of unity.  Fundamental restrictions on irreducible relations among roots of unity were
obtained by Mann \cite{mann1965linear}.  This
viewpoint was applied to Diophantine problems by Newman
\cite{newman1969some} and developed systematically for rational trigonometric
equations by Conway-Jones \cite{conway1976trigonometric}.  It has since been used in several related
trigonometric Diophantine problems: Myerson \cite{myerson1993rational} studied rational
products of sines of rational angles; Laczkovich \cite{laczkovich2003configurations} applied
such equations to configurations with rational angles; Dvornicich-Veneziano-Zannier \cite{DVZ22} used equations in roots of unity in their study of
rational angles in plane lattices; and Kedlaya-Kolpakov-Poonen-Rubinstein \cite{KKPR20} solved a higher-dimensional rational-angle problem by
reducing it to a roots-of-unity equation.

In this paper, we apply this perspective to the structured cosine sums
\eqref{eq:intro-CS}.  As a first illustration, we give a root-of-unity proof of
W\l odarski's classification \cite{wlodarski1969cos} of rational four-cosine vanishing sums in Section \ref{sec:Wlo}, using the
low-weight classification of Poonen-Rubinstein
\cite{poonen1998number}.  This also leads to a sharp two-cosine fiber result.

\begin{theoremintr}[Proposition \ref{prop:three-reps}]\label{thm:two cos fib}
Let $c\in \R\setminus\{0\}$ and set
\begin{equation*}
  \mathcal{R}(c)=\Bigl\{\{\alpha,\beta\}\subset [0,\pi]\cap \Q\pi:
  \cos\alpha+\cos\beta=c\Bigr\},
\end{equation*}
where $\{\alpha,\beta\}$ is a multiset.  Then $|\mathcal{R}(c)|\leq 3$.  Moreover,
$|\mathcal{R}(c)|=3$ if and only if $c\in \left\{\pm \frac12,
  \pm\cos\frac{\pi}{5},\pm\cos\frac{2\pi}{5}\right\}$, with the three representations in each case explicitly determined in Proposition~\ref{prop:three-reps}.
\end{theoremintr}
For sums with more cosine terms, however, the vanishing and multiplicity problems
become substantially more complicated.  The main purpose of this paper is to
study these higher-length questions in the structured sums
\eqref{eq:intro-CS} under some conditions on $N$ or $S$.

We now state the main vanishing results for the structured sums
\eqref{eq:intro-CS}.  For a prime $\ell\mid N$, let $C_\ell\leq G:=\Z/N\Z$ be the
unique subgroup of order $\ell$.  Write
\begin{equation*}
  A\left(u;\frac{N}{\ell}\right):=\sigma(u+C_\ell)+\sigma(-u+C_\ell).
\end{equation*}
When $N=p^aq^b$ with distinct odd primes $p<q$, put
\begin{equation*}
  D_{p,q}(N):=\sigma(C_p)+\sigma(C_q).
\end{equation*}
Then $\ev_N(A(u;\frac{N}{\ell}))=\ev_N(D_{p,q}(N))=0$.  Applying the Lam--Leung theory of vanishing sums of roots of unity \cite{lam2000vanishing}, we obtain the following results.

\begin{theoremintr}[Theorem \ref{thm:symmetric-vanish-paqb} and \ref{thm:weight-2p1}]\label{thm:vanishi}
    Let $G=\Z/N\Z$ and $S$ is a multiset on $G$.

\begin{enumerate}[label=(\arabic*)]
  \item Assume $N=p^a q^b$, where $p<q$ are distinct odd primes. Then $C_S(k)=0$ if and only if
  $\sym(kS)$ is a finite sum of blocks of the three
  types
  \[
     A\left(u;\frac{N}{p}\right),\qquad
     A\left(v;\frac{N}{q}\right),\qquad
     D_{p,q}(N),
  \]
  with $u,v\in G$.  

  \item Let $p$ be the smallest prime divisor of a general integer $N$ and assume $|S|=p$. Then $C_S(k)=0$ if and only if $\sym(kS)=A(u;N/p)$ for some $u\in G$. 
\end{enumerate}
\end{theoremintr}
When $N=2^a q^b$ with $q$ an odd prime, there is an analogous symmetric block
criterion, but the statement is more complicated.  The reason is that the
involution $x\mapsto -x$ has fixed points, and some translates of the prime
subgroups become self-conjugate rather than occurring in conjugate pairs.  We record the precise $2^a q^b$-block list in
Remark~\ref{rem:0 2q}.

The second main theme is small-weight Fourier rigidity. By identifying $k\in \Z/N\Z$ with the character 
\[
\chi_k:\Z/N\Z\to \mathbb{C},\quad n\mapsto \zeta_N^n:=e^{2\pi i n/N},
\]
the Fourier transform of a multiset $X$ on $\Z/N\Z$ is the function $\hat X:\Z/N\Z\to \mathbb{C}$ defined by
\[
\hat X(k)
:=
\sum_{u\in G}X(u)\zeta_N^{ku},
\qquad k\in \Z/N\Z.
\]
We have the following rigidity theorem for the small-weight Fourier transform.
\begin{theoremintr}[Theorem \ref{thm:rig}]\label{thm:fourier-rig}
Let $N$ be odd with $p$ being its smallest prime divisor. Let $G = \mathbb{Z}/N\mathbb{Z}$, and consider two multisets $X$ and $Y$ on $G$.
\begin{enumerate}[label=(\arabic*)]
  \item  Assume
$\max\{|X|,|Y|\}\le p$
and $\hat X(1)=\hat Y(1)\neq 0$. Then $X=Y$, or equivalently $\hat X(g)=\hat Y(g)$ for all $g\in G$.

  \item  Assume $X,Y$ are symmetric, $|X|=|Y|\le 2p$ and $\hat X(1)=\hat Y(1)\neq 0$. Then $X=Y$, or equivalently $\hat X(g)=\hat Y(g)$ for all $g\in G$.
\end{enumerate}
\end{theoremintr}

The condition $\hat X(1)=\hat Y(1)$ should be viewed as equality of the primitive Fourier spectra. Indeed, for
each $a\in G^\times$, the Galois automorphism $\sigma_a(\zeta_N)=\zeta_N^a$ satisfies $\sigma_a(\hat X(1))=\hat X(a)$. Hence $\hat X(1)=\hat Y(1)$ implies $\hat X(a)=\hat Y(a)$ for all $a\in G^\times$. Thus the results above say that, for multisets of sufficiently small weight,
equality of the primitive Fourier spectrum is rigid: equality of one non-zero primitive Fourier coefficient implies equality of the multisets under  small-weight hypotheses.

Our proof of Theorem~\ref{thm:fourier-rig} begins by translating the problem into the integral group-ring formalism. Within this framework, we apply a square-free reduction standard in the study of $N$-th roots of unity. Specifically, with $M = \operatorname{rad}(N)$ denoting the radical of $N$, we decompose $\mathbb{Z}/N\mathbb{Z}$ into cosets of the canonical subgroup of order $M$, which reduces the comparison to the square-free level. We then resolve this base level using Chinese-remainder slicing and induction, an approach analogous to the coset-slicing techniques featured in \cite{lenstra1978vanishing} and \cite{lam2000vanishing}.

The vanishing and multiplicity questions of $C_S(k)$ also have a direct spectral interpretation. Let
$G=\Z/N\Z$, and let $ S\subseteq G$ be a symmetric generating set. Then the eigenvalues of the cyclic Cayley graph
$(G,S)$ are
\begin{equation}\label{eq:intro-eigenvalues}
  \mu_k=\sum_{s\in S} \zeta_N^{ks}=C_S(k),
  \qquad k=0,1,\ldots,N-1 .
\end{equation}
This is z special case of the usual character formula for abelian Cayley graphs; see \cite{babai1979spectra}.  

Thus the vanishing and multiplicity questions of $C_S(k)$ characterize, respectively, the zero eigenvalues and the eigenvalue multiplicities of cyclic Cayley graphs. The eigenvalue problem also has some interesting applications and geometric background; see, for example,  \cite{jiang2021equiangular} and \cite{Uhlenbeck1976}.

A direct consequence of Theorem \ref{thm:two cos fib} and Theorem \ref{thm:fourier-rig} is a bound on the multiplicities of nonzero eigenvalues for the cyclic Cayley graph $(\Z/N\Z,S)$.   For any eigenvalue $\mu$, its multiplicity is defined by
\[
   \operatorname{mult}(\mu):=|\left\{k\in\Z/N\Z:\ \mu_k=\mu\right\}|. 
\]
   In the small-support cases considered below, and under the additional hypothesis
that $S$ contains a unit modulo $N$, we obtain the following bounds.

\begin{theoremintr}[Theorem \ref{thm:mle6} and \ref{thm:mleS}]\label{thm:m(u)}
Let $(\Z/N\Z,S)$ be the cyclic Cayley graph with symmetric
generating set $S$. Let $p$ be the smallest prime divisor of $N$. Assume that $|S|\le 2p $ and $S$ contains a unit. Then every nonzero
  eigenvalue $\mu$ satisfies 
  \[
\operatorname{mult}(\mu)\leq \begin{cases}
    |S|, & \text{if }p\ge 3 \\
   
    6,   & \text{if }p=2
\end{cases}
  \] 
  and this bound is optimal.

\end{theoremintr}

Rational eigenvalues have an additional Galois-theoretic structure.  If
$r\in\Q$ , then $\mu_k=r$ depends only on
$d=\gcd(k,N)$(see Section \ref{sec:rational} for details).  Consequently,
\begin{equation}\label{eq:rational}
    \operatorname{mult}(r)=
   \sum_{\substack{d\mid N\\ \mu_d=r}}
   \varphi\!\left(\frac{N}{d}\right)
\end{equation}
where $\varphi$ is the Euler's totient function.  Combining \eqref{eq:rational} with Theorem~\ref{thm:m(u)} yields an irrationality consequence for cosine sums (see Corollary \ref{cor:mu1-zero-or-irrational}).

Finally, we consider the special case where $N$ is square-free and the generating set $S$ is
a subgroup of the unit group $U_N=(\Z/N\Z)^\times$. For each divisor $m\mid N$, let $U_m=(\Z/m\Z)^\times$, and let
$S_m\leq U_m$ be the image of $S$ under the canonical projection $U_N \twoheadrightarrow U_m$. For
$c\in U_m$, put
\[
   \eta_{m,c}:=\sum_{u\in cS_m}\zeta_m^u
\]
The sums $\eta_{m,c}$ are Gaussian periods.  We also decompose the indices $k\in \Z/N\Z$ into layers
\[
   \Omega_m=\left\{k\in \Z/N\Z:\frac{N}{\gcd(k,N)}=m\right\}.
\]For an eigenvalue $\lambda$ and each divisor $m\mid N$, we define the $m$-layer multiplicity of $\lambda$ by
\[
\operatorname{mult}_m(\lambda):=
\left|\left\{k\in \Omega_m:\mu_k=\lambda\right\}\right|.
\]
Define $E_m:=\{\mu_k: k\in\Omega_m\}$ to be the set of eigenvalues in the $m$-th layer. We adopt the convention that $U_1=S_1=\{1\}$, $\eta_{1,1}=\zeta_1=1$, and $\Omega_1=\{0\}$.

\begin{theoremintr}[Theorems~\ref{thm:squarefree-subgroup-spectrum}
and~\ref{thm:layer-criterion}]
\label{thm:sf}
Assume that $N$ is square-free and $S\leq U_N$.  Then the spectrum of
$(\Z/N\Z,S)$ is described as follows.

\begin{center}
{\normalfont\small
\renewcommand{\arraystretch}{1.25}
\begin{tabularx}{\linewidth}{@{}cXccc@{}}
\toprule
\textbf{Layer}
&
\textbf{Eigenvalue set $E_m$}
&
\textbf{Number}
&
\textbf{Layer multiplicity}
&
\textbf{Degree over $\Q$}
\\
\midrule
$\Omega_m,\ m\mid N$
&
$\displaystyle
\{\frac{|S|}{|S_m|}\eta_{m,c}:cS_m\in U_m/S_m\}$
&
$\displaystyle \frac{\varphi(m)}{|S_m|}$
&
$\displaystyle |S_m|$
&
$\displaystyle \frac{\varphi(m)}{|S_m|}$
\\
\bottomrule
\end{tabularx}}
\end{center}
Moreover, For divisors $m,n>1$, we have either $E_m=E_n$ or $E_m\cap E_n=\emptyset$.
\end{theoremintr}

The paper is organized as follows. Section~\ref{sec:prelim-cosines} recalls the group-ring formalism for vanishing sums of roots of unity, discusses rational cosine sums, provides a roots-of-unity proof of W{\l}odarski's theorem, and proves Theorem~\ref{thm:two cos fib}. Section~\ref{sec:zero} establishes the vanishing criteria presented in Theorem~\ref{thm:vanishi}. Section~\ref{sec:rigidity} is devoted to the proof of Theorem~\ref{thm:fourier-rig}. Finally, Section~\ref{sec:app cay} applies these algebraic results to cyclic Cayley graphs, deriving the zero-eigenvalue criteria, the multiplicity bounds of Theorem~\ref{thm:m(u)}, and the description of square-free unit-subgroup spectra in Theorem~\ref{thm:sf}.

\section{Vanishing sum of roots of unity and W{\l}odarski's theorem}\label{sec:prelim-cosines}

This section collects the background used throughout the paper.  We first recall
vanishing sums of roots of unity and the integral group-ring formalism.  We then
turn to rational cosine sums, give a root-of-unity proof of W{\l}odarski's
four-cosine theorem, and introduce the structured sums $C_S(k)$.

\subsection{Vanishing sums of roots of unity and group ring formalization}\label{sec:vanishing-roots}

We first recall the basic facts on vanishing sums of roots of unity and the
integral group-ring notation used throughout the paper.

\begin{definition}\label{def:vanishing-minimal}
A \emph{vanishing sum of roots of unity} is a formal relation $M=\sum\limits_{i=1}^r a_i \eta_i = 0$, where $a_i\in\Z_{>0}$ and each $\eta_i$ is a root of unity.
The vanishing sum $M$ is called \emph{minimal} if it cannot be written as two non-trivial vanishing sums of roots of unity, i.e., if 
	\[
	\sum_{i=1}^{k}b_i\eta_{i}=0,\quad a_i\ge b_i\ge 0
	\]
	implies either $b_i=a_i$ for all $i$ or $b_i=0$ for all $i$.

Its \emph{weight} is $\eps(M):=\sum\limits_{i=1}^r a_i$. For $N\ge 1$, let $W(N)$ be the set of weights of vanishing sums of $N$-th roots of unity.
\end{definition}

\begin{example}\label{example:vanishing sum}
	Let $p$ be a prime and $\zeta_p$ be a primitive $p$-th root of unity of weight $p$. The vanishing sum $R_p:=1+\zeta_p+\zeta_p^2+\cdots+\zeta_p^{p-1}$ is a minimal vanishing sum of roots of unity. So is any rotation $\zeta R_p$, obtained by multiplying by a root of unity $\zeta$. This type of minimal vanishing sum is called symmetric. A non symmetric minimal vanishing sum is called asymmetric. For example, if $p,q,r$ are three distinct primes, then the sum
	\[
	(\zeta_p+\cdots+\zeta_p^{p-1})(\zeta_q+\cdots+\zeta_q^{q-1})+\zeta_r+\cdots+\zeta_r^{r-1}
	\]
	is an asymmetric vanishing sum of roots of unity of weight $(p-1)(q-1)+(r-1)$.
\end{example}

We have the following theorem to characterize minimal vanishing sums of roots of unity.

\begin{theorem}\label{thm:LamLeung}
Let $N=p_1^{r_1}\cdots p_s^{r_s}$ with distinct primes $p_1<\cdots<p_s$ and $r_i\in \Z_{>0}$.
\begin{enumerate}[label=(\arabic*)]
\item{\rm \cite[Corollary 3.4]{lam2000vanishing}} If $s=2$, then up to rotations, the only minimal vanishing sums of $N$-th roots of unity are $R_{p_1}$ and $R_{p_2}$, i.e., they are symmetric.
\item{\rm(\cite[Theorem 4.8]{lam2000vanishing})} Assume $s\ge 3$. If $M$ is a minimal vanishing sum of $N$-th roots of unity, then either $M$ is symmetric or $\varepsilon(M)\ge (p_1-1)(p_2-1)+(p_3-1)$.

\item{\rm(\cite[Theorem 5.2]{lam2000vanishing})} The set of weights satisfies $W(N)=\N p_1+\cdots+\N p_s$.

\end{enumerate}
\end{theorem}

It is convenient to encode multisets as nonnegative elements of the integral group ring
$\Z[\Z/N\Z]$, and to view cyclotomic sums as the image of an evaluation map
$\ev_N:\Z[\Z/N\Z]\to \Z[\zN]$. We introduce the group-ring formalism
and record the kernel structure that will be used in later sections.

\begin{definition}\label{def:grp ring}
    Let $N \ge 1$ and let $G := \mathbb{Z}/N\mathbb{Z}$ (written additively). Let $\mathbb{Z}[G]$ denote its integral group ring, with basis elements $[x]$ for $x \in G$. An arbitrary element $F \in \mathbb{Z}[G]$ can be written as $F = \sum\limits_{x \in G} c(x)[x]$ with coefficients $c(x) \in \mathbb{Z}$.

(1) We define the involution  $F^*:=\sum\limits_{x\in G}c(x)[-x]$. We say that $F$ is symmetric if $F^* = F$, or equivalently, $c(x) = c(-x)$ for all $x$. We denote $\operatorname{Sym}(F) = F + F^*\in \Z[G].$

(2) The augmentation map is
\[
\eps:\Z[G]\to \Z,\qquad
\sum_x c(x)[x]\mapsto\sum_x c(x).
\]If $F\in \N[G]$, we also call $\eps(F)$ the weight of $F$.

(3) Let $\zeta_N:=e^{2\pi i/N}$. Define the ring homomorphism
\[
\operatorname{ev}_N:\Z[G]\to \Z[\zeta_N],\qquad
\operatorname{ev}_N\!\Big(\sum_x c(x)[x]\Big)=\sum_x c(x)\,\zeta_N^x.
\]

(4) For a finite multiset $T$ in $G$, define its incidence element by $\sigma(T):=\sum\limits_{t\in T}[t]\in \N[G]$. We freely identify $T$ with $\sigma(T)$ whenever convenient.

(5) Let $F^+:=\sum\limits_{x\in G} c^{+}(x)[x]$, $F^-:=\sum\limits_{x\in G} c^{-}(x)[x]\in\N[G]$ be the positive/negative parts of $F$, i.e., $c^+(x)=\max(c(x),0)$, $c^-(x)=\max(-c(x),0)$. We have $F=F^+-F^-$.

(6) We can define a partial order on $\Z[G]$ as follows:

\[
F=\sum_{x\in G} c(x)[x]\ge H=\sum_{x\in G} d(x)[x]
\]
if $c(x)\ge d(x)$ for all $x\in G$.

(7) We define the support of $F$ as $\operatorname{supp}(F) := \{x \in G :\  c(x) \neq 0\}$ and let $\eps_0(F) := |\operatorname{supp}(F)|$. Two elements $F$ and $H$ in $\Z[G]$ are said to be disjoint if $\operatorname{supp}(F) \cap \operatorname{supp}(H) = \emptyset$.
\end{definition}

\begin{remark}\label{rem:symmetrized-odd}
An element $X\in\N[G]$ is of the form $X=\sym(Y)$ for some $Y\in\N[G]$ if and only if $X=X^*$ and the coefficients of all fixed points of $x\mapsto -x$ are even.
Equivalently, if $N$ is odd this means that $X$ is symmetric and $\eps(X)$ is even, while if $N$ is even it means that $X$ is symmetric and the coefficients of $[0]$ and $[N/2]$ are even.
\end{remark}

In the group-ring formalization, a vanishing sum of $N$-th roots of unity can be identified with an element $X \in \mathbb{N}[G] \cap \ker(\ev_N)$. We recall a Theorem of R\'edei \cite{redei1950beitrag}, de Bruijn \cite{debruijn1953factorization} and Schoenberg \cite{schoenberg1964note} that describe the structure of $\ker(\operatorname{ev}_N)$; Lam--Leung \cite[Theorem 2.2]{lam2000vanishing} reproved it using group-ring methods.

\begin{theorem}\label{thm:ker ev}
{\rm (\cite[Theorem 2.2]{lam2000vanishing})} Let $N\ge 1$ and $G=\Z/N\Z$.
For each prime $\ell\mid N$, let $C_\ell\le G$ be the unique subgroup of order $\ell$.
Then $\ker(\operatorname{ev}_N)=\sum\limits_{\ell\mid N}\Z[G]\cdot \sigma(C_\ell)$.

\end{theorem}

We conclude this preliminary subsection by recording a useful comparison theorem of Lam–Leung
for square-free $N$, which will be used later in Section \ref{sec:rigidity}.
\begin{theorem}\label{thm:lam-leung-comparison}{\rm(\cite[Theorem 4.1, Corollary 4.7]{lam2000vanishing})}
Let $N=p_1\cdots p_s$ be square-free with $s\ge 2$ and $p_1<\cdots<p_s$ primes, and let $G=\Z/N\Z$.
Suppose that $x,y\in\N[G]$ satisfy $\ev_N(x)=\ev_N(y)$. Let $\nu(z)$ denote either $\eps_0(z)$ or $\eps(z)$.
If $\nu(x)\le p_1-1$, then either $\nu(y)\ge \bigl(p_1-\nu(x)\bigr)(p_2-1)$ or $y\ge x$.
\end{theorem}

\subsection{Rational cosine sums and W{\l}odarski's theorem}\label{sec:Wlo} Given
angles $\alpha_1,\dots,\alpha_m\in\Q\pi$, the \emph{vanishing problem} asks
when $\sum\limits_{j=1}^m\cos\alpha_j=0$, while the \emph{multiplicity problem} asks
how many multisets of rational angles can give the same fixed nonzero value.

Already in the four-term case, vanishing sums of rational cosines are completely
classified by W{\l}odarski \cite{wlodarski1969cos}.  As an application of vanishing sums of
roots of unity, we give an alternative proof of W{\l}odarski's theorem, based on
the low-weight classification of Poonen--Rubinstein
\cite{poonen1998number}.  We first recall the relevant notation and the
Poonen--Rubinstein classification result.

\begin{notation}
 If $S,T_1,\dots,T_j$ are
vanishing sums, then $(S:T_1,\dots,T_j)$ denotes the vanishing sum obtained by
rotating each $T_i$ so that it shares exactly one distinct root of unity with $S$,
subtracting these $T_i$ from $S$, and absorbing the minus signs into the roots.
Thus $(R_5:R_3)$ is a pentagon sum with one triangle sum subtracted,
$(R_5:3R_3)$ is a pentagon sum with three such subtractions, and
$(R_7:R_3)$ is a heptagon sum with one such subtraction. Finally, $mT$ denotes a
sum of $m$ minimal vanishing sums of the form $T$, with the $m$ summands allowed
to be rotated independently.
\end{notation}

\begin{proposition}[\cite{poonen1998number}]\label{prop:PR}
Let $S$ be a vanishing sum of $8$ roots of unity which is stable under complex
conjugation. Then $S$ can be written in one of the following six forms:
\[
4R_2,\qquad 2R_3+R_2,\qquad R_5+R_3,\qquad (R_5:R_3)+R_2,\qquad
(R_5:3R_3),\qquad (R_7:R_3).
\]
Moreover, the decomposition may be chosen so that each minimal summand is either
itself stable under complex conjugation or occurs together with its complex conjugate.
Finally, if a minimal summand has the form $(R_p:T_1,\dots,T_j)$ with $p\ge 5$,
then, whenever that summand is stable under complex conjugation, the underlying
rotated copy of $R_p$ is also stable under complex conjugation.
\end{proposition}

\begin{proof}
 Any vanishing sum can be
written as a sum of minimal vanishing sums. The complete list of minimal vanishing sums of weight at most $12$ is given by
\cite[Theorem~3 and Table~1]{poonen1998number}; restricting to total weight $8$
gives exactly the six displayed forms. The second assertion is \cite[Lemma~4]{poonen1998number},
and the last one is \cite[Lemma~5]{poonen1998number}.
\end{proof}

\begin{theorem}[W{\l}odarski \cite{wlodarski1969cos}]\label{thm:Wlodarski}
Let $\alpha_1,\alpha_2,\alpha_3,\alpha_4\in[0,\pi]\cap\mathbb Q\pi$ satisfy
\[
\cos\alpha_1+\cos\alpha_2+\cos\alpha_3+\cos\alpha_4=0.
\]
Then the multiset $\{\alpha_1,\alpha_2,\alpha_3,\alpha_4\}$ is one of the following:
\begin{tasks}[label=(\roman*), label-width=2em](2)
\task* $\{\alpha,\beta,\pi-\alpha,\pi-\beta\}$ for some $\alpha,\beta\in[0,\pi]\cap\mathbb Q\pi$;
\task $\{\delta,\tfrac{2\pi}{3}-\delta,\tfrac{2\pi}{3}+\delta,\tfrac{\pi}{2}\}$ for some $0\le \delta\le \tfrac{\pi}{3}$;
\task $\{\tfrac{2\pi}{5},\tfrac{4\pi}{5},\tfrac{\pi}{2},\tfrac{\pi}{3}\}$ \;or\; $\{\tfrac{\pi}{5},\tfrac{3\pi}{5},\tfrac{\pi}{2},\tfrac{2\pi}{3}\}$;
\task $\{\tfrac{\pi}{5},\tfrac{3\pi}{5},\tfrac{\pi}{3},\pi\}$ \;or\; $\{\tfrac{4\pi}{5},\tfrac{2\pi}{5},\tfrac{2\pi}{3},0\}$;
\task $\{\tfrac{2\pi}{5},\tfrac{7\pi}{15},\tfrac{13\pi}{15},\tfrac{\pi}{3}\}$ \;or\; $\{\tfrac{3\pi}{5},\tfrac{8\pi}{15},\tfrac{2\pi}{15},\tfrac{2\pi}{3}\}$;
\task $\{\tfrac{\pi}{15},\tfrac{11\pi}{15},\tfrac{4\pi}{5},\tfrac{\pi}{3}\}$ \;or\; $\{\tfrac{14\pi}{15},\tfrac{4\pi}{15},\tfrac{\pi}{5},\tfrac{2\pi}{3}\}$;
\task $\{\tfrac{2\pi}{7},\tfrac{4\pi}{7},\tfrac{6\pi}{7},\tfrac{\pi}{3}\}$ \;or\; $\{\tfrac{5\pi}{7},\tfrac{3\pi}{7},\tfrac{\pi}{7},\tfrac{2\pi}{3}\}$.
\end{tasks}
\end{theorem}

\begin{proof}
Set $S:=\sum\limits_{j=1}^4(e^{i\alpha_j}+e^{-i\alpha_j})=0$. By Proposition \ref{prop:PR},
$S$ has one of the six forms listed there.

If an $R_2$-summand is $\{u,-u\}$ with $u\neq \pm i$, then two of the angles are
complementary; hence $4R_2$ gives case \textup{(i)}, and outside case \textup{(i)}
every $R_2$ contributes the angle $\pi/2$.

If $S$ has form $2R_3+R_2$, then the two cubic summands are $uR_3$ and
$\bar uR_3$ by Proposition \ref{prop:PR}; writing $u=e^{i\delta}$ with $0\le \delta\le \pi/3$, we obtain
\[
\left\{\delta,\frac{2\pi}{3}-\delta,\frac{2\pi}{3}+\delta,\frac{\pi}{2}\right\},
\]
namely case \textup{(ii)}.

If $S$ has form $R_5+R_3$, then both summands are individually stable under
complex conjugation by Proposition \ref{prop:PR}, so up to sign they are
\[
\{1,\zeta_5^{\pm1},\zeta_5^{\pm2}\},\qquad \{1,\zeta_3,\zeta_3^2\},
\]
with the same sign. This gives exactly the two multisets in case \textup{(iv)}.

If $S$ has form $(R_5:R_3)+R_2$, then the unique weight-$6$ sum of the form
$(R_5:R_3)$ is, up to sign and rotation,
\[
\zeta_6+\zeta_6^{-1}+\zeta_5+\zeta_5^2+\zeta_5^3+\zeta_5^4=0,
\]
so adjoining $\pi/2$ gives exactly the two multisets in case \textup{(iii)}.

If $S$ has form $(R_7:R_3)$, then Proposition \ref{prop:PR} implies that the underlying
heptagon is stable under complex conjugation; hence, up to sign and rotation,
\[
-\zeta_3-\zeta_3^2+\zeta_7+\zeta_7^2+\zeta_7^3+\zeta_7^4+\zeta_7^5+\zeta_7^6=0,
\]
which yields exactly the two multisets in case \textup{(vii)}.

Finally, if $S$ has form $(R_5:3R_3)$, then Proposition \ref{prop:PR} implies that the
underlying pentagon is stable under complex conjugation. Hence, after sign and
rotation, the two untouched vertices form a conjugation-stable $2$-subset of
$\mu_5$, so up to rotation they are $\{\zeta_5^{\pm1}\}$ or $\{\zeta_5^{\pm2}\}$. These
yield respectively
\[
\left\{\frac{2\pi}{5},\frac{7\pi}{15},\frac{13\pi}{15},\frac{\pi}{3}\right\},
\qquad
\left\{\frac{\pi}{15},\frac{11\pi}{15},\frac{4\pi}{5},\frac{\pi}{3}\right\},
\]
and multiplying the whole sum by $-1$ gives the companion multisets in
cases \textup{(v)} and \textup{(vi)}.
\end{proof}

\begin{remark}
    W{\l}odarski's original proof is a direct and remarkably ingenious analysis of the
trigonometric equation. The proof above gives a different explanation from the
roots-of-unity viewpoint: it rewrites the identity as a
conjugation-stable vanishing sum of eight roots of unity and  applies the
Poonen--Rubinstein low-weight classification.
\end{remark}

As a consequence, we obtain a sharp bound for the two-term multiplicity problem
for rational cosine sums.

\begin{proposition}\label{prop:three-reps}
Fix $c\in\mathbb R\setminus\{0\}$ and define
\[
\mathcal R(c)
:=\Big\{\{\alpha,\beta\}\subset[0,\pi]\cap\mathbb Q\pi:\ \cos\alpha+\cos\beta=c\Big\}
\]where $\{\alpha,\beta\}$ is a multiset (allow $\alpha=\beta$).
Then $|\mathcal R(c)|\le 3$ and at most one pair in
$\mathcal R(c)$ has equal reduced denominators. Moreover,  $|\mathcal R(c)|=3$ if and only if 
$c\in\{\pm\frac12,\ \pm\cos\frac{\pi}{5},\ \pm\cos\frac{2\pi}{5}\}$. In each case $\mathcal R(c)$ is \emph{exactly} the following triple of pairs:
\begin{align*}
\mathcal R\Big(\frac12\Big)
&=\Big\{\Big\{0,\frac{2\pi}{3}\Big\},\ \Big\{\frac{\pi}{5},\frac{3\pi}{5}\Big\},\ 
\Big\{\frac{\pi}{2},\frac{\pi}{3}\Big\}\Big\},\\[2mm]
\mathcal R\Big(-\frac12\Big)
&=\Big\{\Big\{\pi,\frac{\pi}{3}\Big\},\ \Big\{\frac{2\pi}{5},\frac{4\pi}{5}\Big\},\ 
\Big\{\frac{\pi}{2},\frac{2\pi}{3}\Big\}\Big\},\\[2mm]
\mathcal R\Big(\cos\frac{\pi}{5}\Big)
&=\Big\{\Big\{\frac{\pi}{5},\frac{\pi}{2}\Big\},\ \Big\{\frac{2\pi}{5},\frac{\pi}{3}\Big\},\ 
\Big\{\frac{2\pi}{15},\frac{8\pi}{15}\Big\}\Big\},\\[2mm]
\mathcal R\Big(\cos\frac{2\pi}{5}\Big)
&=\Big\{\Big\{\frac{2\pi}{5},\frac{\pi}{2}\Big\},\ \Big\{\frac{\pi}{5},\frac{2\pi}{3}\Big\},\ 
\Big\{\frac{\pi}{15},\frac{11\pi}{15}\Big\}\Big\},\\[2mm]
\mathcal R\Big(-\cos\frac{\pi}{5}\Big)
&=\Big\{\Big\{\frac{4\pi}{5},\frac{\pi}{2}\Big\},\ \Big\{\frac{3\pi}{5},\frac{2\pi}{3}\Big\},\ 
\Big\{\frac{7\pi}{15},\frac{13\pi}{15}\Big\}\Big\},\\[2mm]
\mathcal R\Big(-\cos\frac{2\pi}{5}\Big)
&=\Big\{\Big\{\frac{3\pi}{5},\frac{\pi}{2}\Big\},\ \Big\{\frac{4\pi}{5},\frac{\pi}{3}\Big\},\ 
\Big\{\frac{4\pi}{15},\frac{14\pi}{15}\Big\}\Big\}.
\end{align*}
\end{proposition}

\begin{proof}
Fix $\{\alpha,\beta\}\in\mathcal R(c)$.  Any other $\{\gamma,\delta\}\in\mathcal R(c)$ satisfies
\[
\cos\alpha+\cos\beta=\cos\gamma+\cos\delta
\quad\Longleftrightarrow\quad
\cos\alpha+\cos\beta+\cos(\pi-\gamma)+\cos(\pi-\delta)=0.
\]
Hence the multiset $Q(\gamma,\delta):=\{\alpha,\beta,\pi-\gamma,\pi-\delta\}$ is a rational $4$-cosine vanishing sum.  By Theorem~\ref{thm:Wlodarski}, it must be of one
of the types \textup{(i)}--\textup{(vii)}.

If $Q(\gamma,\delta)$ is of type \textup{(i)}, since $c=\cos \alpha+\cos \beta\ne 0$, we have $\alpha\ne \pi-\beta$. Hence 
$\{\pi-\gamma,\pi-\delta\}=\{\pi-\alpha,\pi-\beta\}$ and thus $\{\gamma,\delta\}=\{\alpha,\beta\}$;
this is the trivial representation.

In the nontrivial cases \textup{(ii)}--\textup{(vii)}, the multiset $Q(\gamma,\delta)$ is one of
a one-parameter family \textup{(ii)} or one of finitely many sporadic quadruples
\textup{(iii)}--\textup{(vii)}.
For a fixed pair $\{\alpha,\beta\}$:
\begin{itemize}
\item it can occur in \emph{at most one} quadruple of type \textup{(ii)} since the parameter $\delta$ is
forced once two angles of the multiset are specified.
\item by direct inspection of sporadic quadruples \textup{(iii)}--\textup{(vii)}, a fixed 2-element
subset of angles can be contained in \emph{at most two} sporadic quadruples.
Moreover, the only instances where two distinct sporadic quadruples share \emph{two} angles are:
\[
\Big\{\frac{2\pi}{5},\frac{4\pi}{5}\Big\},\ \Big\{\frac{2\pi}{5},\frac{\pi}{3}\Big\},\ 
\Big\{\frac{4\pi}{5},\frac{\pi}{3}\Big\},\ 
\Big\{\frac{\pi}{5},\frac{3\pi}{5}\Big\},\ 
\Big\{\frac{3\pi}{5},\frac{2\pi}{3}\Big\},\ 
\Big\{\frac{\pi}{5},\frac{2\pi}{3}\Big\}.
\]

\item if $\{\alpha,\beta\}$ lies in two sporadic quadruples, it lies in no type (ii) quadruple.
\end{itemize}
Therefore $\{\alpha,\beta\}$ can be extended to at most two distinct nontrivial W{\l}odarski quadruples,
and each such extension determines at most one distinct new pair $\{\gamma,\delta\}$.
Hence there are at most two nontrivial representations besides $\{\alpha,\beta\}$, proving
$|\mathcal R(c)|\le 3$.  Moreover, inspection
of Włodarski’s list shows that for fixed $c\neq0$, at most one pair in
$\mathcal R(c)$ has equal reduced denominators.

For the equality case, assume $|\mathcal R(c)|=3$. Write $\mathcal R(c)=\{P_1,P_2,P_3\}$ and $P_i=\{\alpha_i,\beta_i\}$. For $i\neq j$, let $Q_{ij}:=\{\alpha_i,\beta_i,\pi-\alpha_j,\pi-\beta_j\}$. Call the relation between $P_i$ and $P_j$ \emph{sporadic} or \emph{type \textup{(ii)}} according as
$Q_{ij}$ is sporadic or type \textup{(ii)}. This is a symmetric relation. Now fix $i$. The two equalities between $P_i$ and the other two elements of $\mathcal R(c)$ give two
nontrivial Włodarski quadruples containing $P_i$. By the first bullet above, a fixed
pair can lie in at most one type \textup{(ii)} quadruple. Hence at least one of the two relations
from $P_i$ is sporadic. Thus each of $P_1,P_2,P_3$ is incident to at least one sporadic relation. Therefore there are at
least two sporadic relations among the three pairs $\{P_1,P_2\},\{P_1,P_3\},\{P_2,P_3\}$. Hence some $P_i$ lies in two distinct sporadic quadruples. So some $P_i$ must to be one of the six shared sporadic pairs listed in the second bullet above. For each of these six possibilities, the two sporadic quadruples containing it are unique; reading
off the complementary pairs gives the other two elements of $\mathcal R(c)$, hence exactly the six
triples listed in the statement.
\end{proof}

We now introduce the structured cosine sums that will be the main object of the paper.

\begin{definition}\label{def:C_S(k)}
    Fix an integer $N\ge 1$, and let $S=\{s_1,\dots,s_t\}$ be a multiset in $\Z/N\Z$.
For each $k\in\Z/N\Z$, we define the structured  cosine sums
\[
C_S(k)=\sum_{j=1}^t \cos\!\Big(\frac{2\pi ks_j}{N}\Big).
\]
\end{definition}

Using the group-ring notation from the previous subsection and identifying a multiset
with its incidence element, we have
   $2C_S(k)=\ev_N(\sym(kS))$. Thus vanishing and equality questions for $C_S(k)$ become vanishing and
equality questions for cyclotomic evaluations with nonnegative coefficients. The two basic questions are the same as in the  rational cosine sums:
\emph{(vanishing)} for which $k$ does $C_S(k)=0$, and \emph{(multiplicity)} how large can a fiber $\{k:C_S(k)=c\}$ be for a fixed $c$?

\begin{remark}
    After clearing denominators, any finite sum of cosines with
angles in $\Q\pi$ can be rewritten in the form $C_S(1)$ for a suitable modulus $N$.
\end{remark}

\section{Symmetric vanishing multisets and zeros of structured cosine sums}\label{sec:zero}

We now investigate the vanishing problem for structured cosine sums. 

\begin{definition}\label{def:blocks}
Let $p\mid N$ be prime and let $u\in G$.
Define
\[
A\!\Big(u;\frac{N}{p}\Big):=\sym(u+C_p)=\sigma(u+C_p)+\sigma(-u+C_p)\in \N[G],
\]
where $C_p\le G$ is the unique subgroup of order $p$.
Equivalently, as a multiset, 
\[
A\!\Big(u;\frac{N}{p}\Big)=\Big\{\pm\big(u+j\frac{N}{p}\big):j=0,1,\dots,p-1\Big\}.
\]
If $N$ has two distinct odd primes $p\ne q$, define
\[
D_{p,q}(N):=\sigma(C_p)+\sigma(C_q)\in \N[G].
\]
As a multiset,
\[
D_{p,q}(N)=\Big\{j\frac{N}{p}:j=0,\dots,p-1\Big\}
+
\Big\{j\frac{N}{q}:j=0,\dots,q-1\Big\}.
\]
As throughout, these may be read either as multisets or as elements of $\N[G]$.
\end{definition}
Note that $\ev_N\!\Big(A\!\big(u;\frac{N}{p}\big)\Big)=\zN^uR_p+\zN^{-u}R_p=0$ and $\ev_N\big(D_{p,q}(N)\big)=R_p+R_q=0$. They are typical vanishing sums of cosine functions.

In the case $N=p^aq^b$ with distinct primes, symmetric vanishing multisets admit an explicit decomposition.

\begin{theorem}\label{thm:symmetric-vanish-paqb}
Let $N=p^aq^b$ with distinct odd primes $p<q$, and let $X\in\N[G]$.
Then the following are equivalent.
\begin{enumerate}[label=(\arabic*)]
\item $X=\sym(T)$ for some $T\in \N[G]$ and $\ev_N(X)=0$ (we call such $X$ a \emph{symmetric vanishing multiset}).

\item $X$ is a finite sum of blocks of the three types
\[
A\!\Big(u;\frac{N}{p}\Big),\qquad A\!\Big(v;\frac{N}{q}\Big),\qquad D_{p,q}(N),
\]
for some suitable $u,v\in G$.
\end{enumerate}
\end{theorem}

\begin{proof}

The implication (2) $\Rightarrow$ (1) is clear. For (1) $\Rightarrow$(2), we argue by induction on $\eps(X)$. Note that by Remark \ref{rem:symmetrized-odd}, (1) is equivalent to $X$ is symmetric with $\eps(X)$ even and $\ev_N(X)=0$.
If $\eps(X)=0$, there is nothing to prove.
Assume $\eps(X)>0$.
Choose a nonzero element $Y\in\N[G]$ such that
\[
0<Y\le X,\qquad \ev_N(Y)=0,
\]
and $\eps(Y)$ is minimal among all such choices.
Then $Y$ is a minimal vanishing sum, so by Theorem~\ref{thm:LamLeung}(1) it is of the form $Y=\sigma(u+C_p)$ or $Y=\sigma(v+C_q)$ for some $u,v\in G$.

We take $Y=\sigma(u+C_p)$ for example and the similar argument works for $Y=\sigma(v+C_q)$.
If $u\notin C_p$, then $Y$ and $Y^*$ have disjoint support.
Indeed, if $x\in\operatorname{supp}(Y)\cap\operatorname{supp}(Y^*)$, then $u+j\frac{N}{p}\equiv -u+j'\frac{N}{p}\pmod N$ for some $j,j'$, so $2u\in C_p$.
Since $N$ is odd, $2$ is invertible modulo $N$, hence $u\in C_p$, a contradiction.
Because $X$ is symmetric and $Y\le X$, we also have $Y^*\le X$.
Therefore, if $u\notin C_p$, we have $A\!(u;\frac{N}{p})=Y+Y^*\le X$. Set $X':=X-A\!(u;\frac{N}{p})$. Then $X'\in\N[G]$ is still symmetric, still has even weight, and still satisfies $\ev_N(X')=0$.
By the induction hypothesis, $X'$ is a finite sum of the required blocks, hence so is $X$. 

Iterating this process for $p$ and $q$, we are reduced to the case $X=e\sigma(C_p)+f\sigma(C_q)$ for some $e,f\in\N$.
Because $\eps(X)$ is even and $p,q$ are odd, $e+f$ is even.
Therefore at least one of the following holds:
\[
e\ge 2,\qquad f\ge 2,\qquad \text{or}\qquad e\ge 1\text{ and }f\ge 1.
\]
If $e\ge 2$, then $A\!(0;\frac{N}{p})=2\sigma(C_p)\le X$. If $f\ge 2$, then $A\!(0;\frac{N}{q})=2\sigma(C_q)\le X$. If $e\ge 1$ and $f\ge 1$, then $D_{p,q}(N)=\sigma(C_p)+\sigma(C_q)\le X$. Subtracting one of these blocks produces a new element $X'\in\N[G]$ with smaller weight, still symmetric, still of even weight, and still satisfying $\ev_N(X')=0$.
The induction hypothesis now completes the proof.
\end{proof}

\begin{cor}
Let $N=p^aq^b$ with distinct odd primes $p<q$.
Then $C_S(k)=0$ if and only if $\sym(kS)$ can be written as a finite sum of blocks of the three types in Theorem~\ref{thm:symmetric-vanish-paqb}.
\end{cor}

\begin{proof}
It follows from $2C_S(k)=\ev_N(\sym(kS))$ and Theorem~\ref{thm:symmetric-vanish-paqb}.
\end{proof}

\begin{remark}\label{rem:0 2q}
Let $N=2^aq^b$ with $q$ an odd prime.
The same inductive argument gives the analogue of Theorem~\ref{thm:symmetric-vanish-paqb}.
Every symmetric vanishing element of $\N[G]$ is a finite sum of blocks of the forms
\[
(u+C_2)+(-u+C_2),\qquad (v+C_q)+(-v+C_q),
\]
where the two summands are distinct, together with the self-conjugate blocks
\[
C_2=\Big\{0,\frac{N}{2}\Big\},\qquad
\frac{N}{4}+C_2=\Big\{\frac{N}{4},\frac{3N}{4}\Big\}\ \ (4\mid N),
\]
\[
C_q=\Big\{j\frac{N}{q}:0\le j<q\Big\},\qquad
\frac{N}{2}+C_q=\Big\{\frac{N}{2}+j\frac{N}{q}:0\le j<q\Big\}.
\]
Indeed, a translate $u+C_2$ is fixed by involution $T\mapsto T^*$ if and only if $u+C_2=C_2$, or $4\mid N$ and $u+C_2=\frac{N}{4}+C_2$; similarly, a translate $v+C_q$ is fixed by involution if and only if $v+C_q=C_q$ or $v+C_q=\frac{N}{2}+C_q$.
\end{remark}

For a general $N$, there are minimal vanishing sums of roots of unity which are not rotations of $R_p$ (cf. Example \ref{example:vanishing sum}), and it is difficult to classify all of them for larger weight.
So we restrict to the smallest even weight relevant for structured cosine sums.

\begin{theorem}\label{thm:weight-2p1}
Let $N=p_1^{r_1}\cdots p_s^{r_s}$ with distinct primes $p_1<\cdots<p_s$.
Let $X\in\N[G]$ be of the form $X=\sym(Y)$ for some $Y\in\N[G]$, and suppose that $\eps(X)=2p_1$.
Then $\ev_N(X)=0$ if and only if $X=A(u;N/p_1)$ for some $u\in G$.
\end{theorem}

\begin{proof}
The sufficiency is clear. For the necessity, if $s\le 2$, it follows from Theorem \ref{thm:symmetric-vanish-paqb}.
If $s\ge 3$, then by Theorem~\ref{thm:LamLeung}(2), every asymmetric minimal vanishing summand that is not a rotation of some $R_r$ has weight at least $(p_1-1)(p_2-1)+(p_3-1)$. Since $p_2\ge 3$ and $p_3\ge 5$, we have
\[
(p_1-1)(p_2-1)+(p_3-1)\ge 2(p_1-1)+4=2p_1+2>2p_1.
\]
Hence every minimal vanishing summand of $X$ is a rotation of $R_r$, equivalently a translate $u+P_r$. Each such summand has prime weight at least $p_1$, while $\eps(X)=2p_1$.
Therefore $X=\sigma(u+C_{p_1})+\sigma(v+C_{p_1})$ for some $u,v\in G$. Assume first that $p_1\ne 2$, i.e., $N$ is odd, then the same argument as in Theorem \ref{thm:symmetric-vanish-paqb} gives the result.

It remains to consider the case $p_1=2$.
Then $\eps(X)=4$, so $\eps(Y)=2$.
Write $Y=[a]+[b]$ for some $a,b\in G$.
Then $X=[a]+[-a]+[b]+[-b]$. The condition $\ev_N(X)=0$ becomes $\cos\!(\frac{2\pi a}{N})+\cos\!(\frac{2\pi b}{N})=0$. Therefore $b\equiv \frac{N}{2}\pm a \pmod N$, and so $X=A\!(a;\frac{N}{2})$.
\end{proof}

\begin{cor}
Let $N=p_1^{r_1}\cdots p_s^{r_s}$ with distinct primes $p_1<\cdots<p_s$.
Let $S\subset \Z/N\Z$ be a multiset with $|S|=p_1$.
Then $C_S(k)=0$ if and only if $\sym(kS)=A(u;N/p_1)$ for some $u\in G$.
\end{cor}

\section{Small-weight Fourier rigidity}\label{sec:rigidity}
In this section, we prove the rigidity theorem stated in Theorem~\ref{thm:fourier-rig}. First, we pass to the integral group-ring formalism, which is better suited to the proof. Let $N$ be an odd integer, let $G \coloneqq \mathbb{Z}/N\mathbb{Z}$, and let $p$ denote the smallest prime divisor of $N$. A multiset $X$ on $G$ can be identified with its formal sum $\sigma(X) := \sum\limits_{x \in X} [x] \in \mathbb{N}[G]$. In this notation, we have $\varepsilon(X) = |X|$ and $\operatorname{ev}_N(X) = \hat{X}(1)$. Thus, the Fourier-analytic statements in Theorem~\ref{thm:fourier-rig} are precisely the group-ring statements formulated below.

\begin{theorem}\label{thm:rig}
Let $N$ be odd, and let $p$ be the smallest prime divisor of $N$. Let $G:=\Z/N\Z$ and $X,Y\in \N[G]$. 

(1) Assume  $\max\{\eps(X),\eps(Y)\}\le p$ and  $\ev_N(X)=\ev_N(Y)\neq 0$. Then $X=Y$ in $\N[G]$.

(2) Assume  $X,Y$ are symmetric , $\eps(X)=\eps(Y)\le 2p$ and $\ev_N(X)=\ev_N(Y)\neq 0$. Then $X=Y$ in $\N[G]$. 
\end{theorem}

The proof has two steps.  Subsection~\ref{sec:sf case} treats square-free  $N$ by decomposing
$\mathbb Z[\mathbb Z/N\mathbb Z]$ along a prime factor $N=pM$.  Subsection~\ref{sec:general odd} then
lifts the result to general odd $N$ by decomposing $\mathbb Z/N\mathbb Z$ into cosets
of the canonical subgroup of square-free order.

\subsection{Square-free case}\label{sec:sf case}

When $N=p$ has only one odd prime factor, Theorem \ref{thm:rig}  follows from the following Lemma which is a special case of Theorem \ref{thm:ker ev}. 

\begin{lemma}\label{lem:one prime}
Let $p$ be a prime. If $X,Y\in \Z[\Z/p\Z]$ satisfy $\ev_p(X)=\ev_p(Y)$, then $X-Y\in \Z\cdot\sigma(\Z/p\Z)$. In particular, $p\,|\eps(X)-\eps(Y)$.
\end{lemma}

In the rest of this subsection, let $M=p_1p_2\cdots p_s$ be a square-free odd number, where $p_1<p_2<\cdots<p_s$ are primes, and let $U:=\Z/M\Z$. Assume $s\ge 2$, and put $p:=p_1$ and $q:=p_2$. Write $M=pD$, and let $V:=\Z/p\Z$ and $W:=\Z/D\Z$. 
The map  defined by 
\[
\Psi:V\times W\to U\qquad (v,w)\mapsto Dv+pw\pmod M
\] is a group isomorphism since $\gcd
(p,D)=1$. By the induced isomorphisms of group rings $\Z[U]\cong \Z[V\times W]\cong \Z[V]\otimes_{\Z}\Z[W]$, every $F\in \Z[U]$ can be written as
\[
F=\sum_{v\in V}[v]\otimes F_v,
\qquad
F_v\in \Z[W].
\]

\begin{lemma}\label{lem:F_v}
(1) If F is symmetric, then $F_{-v}=F_v^*
$ for every $v\in V$. In particular, $F_0\in \Z[W]$ is symmetric.

(2) We have $\ev_M(F)=\sum\limits_{v\in V}\zeta_p^v\,\ev_D(F_v)$. Let $A,B\in\Z[U]$ with $\ev_M(A)=\ev_M(B)$. Then there exists $S\in\Q(\zeta_D)$, such that $\ev_D(A_v)-\ev_{D}(B_v)=S$ for all $v\in V$.
\end{lemma}

\begin{proof}
We Write $F=\sum\limits_{v\in V}\sum\limits_{w\in W} c(v,w)[(v,w)]$, then $F_v=\sum\limits_{w\in W} c(v,w)[w]$.

(1) If $F$ is symmetric, then $c(v,w)=c(-v,-w)$ for all $(v,w)\in V\times W$.
Hence
\[
F_{-v}
=\sum_{w\in W} c(-v,w)[w]
=\sum_{w\in W} c(v,-w)[w]
=F_v^*.
\]

(2) Since $\ev_M([(v,w)])=\zeta_M^{Dv+pw}=\zeta_p^v\zeta_D^w$, we have
\[
\ev_M(F)
=\sum_{v\in V}\sum_{w\in W} c(v,w)\zeta_p^v\zeta_D^w
=\sum_{v\in V}\zeta_p^v\left(\sum_{w\in W} c(v,w)\zeta_D^w\right)
=\sum_{v\in V}\zeta_p^v\,\ev_D(F_v).
\]
Hence
$0=\ev_M(A-B)=\sum\limits_{v\in V}\zeta_p^v(\ev_D(A_v)-\ev_D(B_v))$. Since $p$ is prime and $\gcd(p,D)=1$, we have $\Q(\zeta_p)\cap \Q(\zeta_D)=\Q$.
 Therefore the cyclotomic polynomial $\Phi_p(T)=1+T+\cdots+T^{p-1}$ is the minimal polynomial of $\zeta_p$ over $\Q(\zeta_D)$. Hence the
coefficients $\ev_D(A_v)-\ev_D(B_v)$ are all equal.
\end{proof}

\begin{proposition}\label{prop:rigid nonsym}
Let $M$ be an odd square-free integer with smallest prime factor $p$, and let $U := \mathbb{Z}/M\mathbb{Z}$. Let $X, Y \in \mathbb{N}[U]$ be such that $\max\{\eps(X),\eps(Y)\}\le p$ and  $\ev_M(X)=\ev_M(Y)\neq 0$.
Then $X=Y$ in $\N[U]$. Furthermore, if $\max\{\eps(X),\eps(Y)\}< p$ and $\ev_M(X)=\ev_M(Y)$, then $X=Y$.
\end{proposition}

\begin{proof}
The final assertion follows from the first part and Theorem \ref{thm:LamLeung}(3). It suffices to prove the first part. Let $Z:=X-Y\in \Z[U]$. We argue by induction on the number $s$ of prime factors of
$M$. The case $s=1$ follows from Lemma \ref{lem:one prime}. Assume now that $s\ge 2$. Using the $\Psi$-identification, we have $Z=\sum\limits_{v\in V}[v]\otimes Z_v$ where $Z_v\in \Z[W]$. For each $v\in V$, let $Z_v^+,Z_v^-\in \N[W]$ be the positive/negative part of $Z_v$. Then we have $0=\ev_M(Z)=\ev_M(Z^+)-\ev_M(Z^-)$ and
\begin{equation}\label{eq:e(Z_v^+)}
    \sum_{v\in V}\eps(Z_v^{+})=\eps(Z^+)\le \eps(X)\le p,
\qquad
\sum_{v\in V}\eps(Z_v^-)=\eps(Z^-)\le \eps(Y)\le p.
\end{equation}
By Lemma \ref{lem:F_v}(2), there exists $S\in \Q(\zeta_D)$ such that
\begin{equation}\label{eq:ev-ev=S1}
    \ev_D(Z_v^+)-\ev_D(Z_v^-)=S\qquad \forall\, v\in V.
\end{equation}
Set $m_v:=\eps(Z_v^+)+\eps(Z_v^-)$. Then by \eqref{eq:e(Z_v^+)},
\begin{equation}\label{eq:sum m_u<2p}
    \sum_{v\in V} m_v\le 2p.
\end{equation}
We consider the following two cases.

\noindent\textbf{Case 1: assume that $m_{v_0}\le 1$ for some $v_0\in V$.} Put $A:=Z_{v_0}^+$ and $B:=Z_{v_0}^-$. Then $A$ and $B$ are disjoint and $\eps(A)+\eps(B)=m_{v_0}\le 1$. After replacing $Z$ by $-Z$ if necessary, we may assume that $A=0$. Then $B$ is either $0$ or a singleton. And by \eqref{eq:ev-ev=S1},
\[
\ev_D(Z_v^++B)=\ev_D(Z_v^-+A)=\ev_D(Z_v^+)\qquad (\forall\, v\in V).
\]
Now $\eps(Z_v^++B)\le p+1<q$ and $\eps(Z_v^-)\le p<q$.  If $\ev_D(Z_v^++B)=\ev_D(Z_v^-)= 0$, then Theorem \ref{thm:LamLeung}(3) implies $Z_v^++B=Z_v^-=0$. Therefore $Z_v^+=Z_v^-=B=0$. If $\ev_D(Z_v^++B)=\ev_D(Z_v^-)\ne 0$, by inductive hypothesis, we have $Z_v^++B=Z_v^-$. Since $Z_v^+$ and $Z_v^-$ are disjoint, this implies $Z_v^+=0$ and $Z_v^-=B$. So we have $Z_v^+=0$ and $Z_v^-=B$ for every $v\in V$. 

If $B=0$, then $Z=0$ and $X=Y$. Otherwise $B=[b]$ for some $b\in W$, and therefore $Z^+=0$ and $Z^-=\sigma(V\times\{b\})$. Since $Z^-\le Y$, $\eps(Z^-)=p$ and $\eps(Y)\le p$, we get $Y=Z^-$. Hence $\ev_M(Y)=\ev_M(\sigma(V\times\{b\}))=\zeta_D^b\sum\limits_{v\in V}\zeta_p^v=0$, a contradiction. Therefore Case 1 implies $X=Y$.

\noindent\textbf{Case 2: assume that $m_v\ge 2$ for every $v\in V$.}  Then \eqref{eq:sum m_u<2p} implies

\begin{equation}\label{eq:m_u=2}
 m_v=2\quad \forall v\in V,
\qquad
\eps(Z^+)=\eps(Z^-)=p.   
\end{equation}
So for each $v$, the pair $(\eps(Z_v^+),\eps(Z_v^-))\in \{(2,0),(1,1),(0,2)\}.$ Let $a,b,c$ be the numbers of  these three types. Then
\[
a+b+c=p,\qquad 2a+b=\eps(Z^+)=p,\qquad b+2c=\eps(Z^-)=p,
\]
hence $a=c$. If $a=c>0$, choose $i$ of type $(2,0)$ and $j$ of type $(0,2)$. Then $\ev_D(Z_i^+)=S$ and $\ev_D(Z_j^-)=-S$, so $\ev_D(Z_i^++Z_j^-)=0$. But $\eps(Z_i^++Z_j^-)=4<q$, impossible by Theorem \ref{thm:LamLeung}(3). Hence $a=c=0$, and  $(\eps(Z_v^+),\eps(Z_v^-))=(1,1)$ for all $v\in V$.

Now for any $i,j\in V$, \eqref{eq:ev-ev=S1} gives $\ev_D(Z_i^++Z_j^-)=\ev_D(Z_i^-+Z_j^+)$. Each side has weight $2<q$, hence $\ev_D(Z_i^++Z_j^-)=\ev_D(Z_i^-+Z_j^+)\ne 0$ by Theorem \ref{thm:LamLeung}(3). By inductive hypothesis, we have $Z_i^++Z_j^-=Z_i^-+Z_j^+$. Thus $Z_i=Z_j$ for all $i,j\in V$. Hence there exist $w_1,w_2\in W$ such that $Z_v=[w_1]-[w_2]$ for all $v\in V$. Therefore $Z^+=\sigma(V\times\{w_1\})$ and $Z^-=\sigma(V\times\{w_2\})$. Now \eqref{eq:e(Z_v^+)} and \eqref{eq:m_u=2} imply $X=Z^+$ and $Y=Z^-$. Then $\ev_M(X)=\ev_M(Y)=0$, a contradiction. Hence case 2 cannot happen. 
\end{proof}

\begin{example}\label{example:rigid nonsym}
(1) The oddness assumption is essential in Proposition \ref{prop:rigid nonsym}. Take 
$M=6$ and $U=\Z/6\Z$. Let $X=[0]+[2]$ and $Y=[1]$. Then $\eps(X),\eps(Y)\le p=2$. We have $X\ne Y$ but $\ev_{6}(X)=1+\zeta_6^2=\zeta_6=\ev_{6}(Y)\ne 0$.

(2) The condition that the common evaluation is nonzero is necessary. Let $M=pq$ with $2<p<q$ primes and $U=\Z/M \Z$. Let $C_p$ be the subgroup of $U$ of order $p$. Take $X=\sigma(C_p)$ and $Y=\sigma(1+C_p)$. Then $\eps(X)=\eps(Y)=p$ and $\ev_{M}(X)=\ev_M(Y)=0$. But $X\ne Y$.

(3) The bound  $p_1$ for $\eps(X),\eps(Y)$  is optimal. Let $(p,q)$ be a twin prime pair, i.e., $q=p+2$. Let $M=pq$ with $2<p<q$ primes and $U=\Z/M \Z$. Take $X=\sigma(C_p\setminus\{0\})$ and $Y=\sigma(C_q\setminus\{0\})$. Then $\eps(X)=p-1$ and $\eps(Y)=q-1=p+1>p$. We have $\ev_{pq}(X)=\ev_{pq}(Y)=-1$ but $X\ne Y$.
\end{example}

We now consider the symmetric case. We begin by showing that a disjoint slice pair of total weight at most $3$ whose difference has real image under $\ev_D$ must have one of only a few very rigid forms. 

\begin{lemma}\label{lem:v-slice}
Let $X,Y\in\N[W]$ be disjoint. Assume that $1\le \eps(X)+\eps(Y)\le 3$ and 
$\ev_D(X)-\ev_D(Y)\in \R$. Then, after interchanging $X$ and $Y$ if necessary, one of the following holds:
\[
(X,Y)=([0],0),\qquad (C,0),\qquad (C,[0]),\qquad (P,0),
\]
where $C=C^*\in\N[W]$ with $\eps(C)=2$, and $P=P^*\in\N[W]$ with $\eps(P)=3$.
\end{lemma}

\begin{proof}
We first note that $(\eps(X),\eps(Y))=(1,1)$ cannot occur: if $X=[x]$ and $Y=[y]$, then $\ev_D([x]-[y])\in\R$. Since $D$ is odd, this implies $x=y$, contradicting disjointness.

If $(\eps(X),\eps(Y))=(1,0)$, then after interchanging $X$ and $Y$ if necessary
we may write $(X,Y)=([a],0)$. Since $\ev_D([a])\in\R$, we have $a=0$, so $(X,Y)=([0],0)$.

If $(\eps(X),\eps(Y))=(2,0)$ or $(3,0)$, then $\ev_D(X)\in\R$, hence $\ev_D(X)=\ev_D(X^*)$. Both $X$ and $X^*$ are positive of weight at most $3<q$. Therefore
Proposition~\ref{prop:rigid nonsym} implies $X=X^*$. Thus $X$ is symmetric,
so $X=C$ or $X=P$ according as $\eps(X)=2$ or $3$.

Finally, assume $(\eps(X),\eps(Y))=(2,1)$. Write $Y=[a]$. Since $\ev_D(X)-\ev_D([a])\in\R$, we have $\ev_D(X+[{-a}])=\ev_D(X^*+[a])$. Both sides are positive of weight $3<q$. Hence Proposition~\ref{prop:rigid nonsym} 
implies $X+[{-a}]=X^*+[a]$. Therefore $a=0$ since $X$ and $Y=[a]$ are disjoint. And furthermore, $X=X^*$.
\end{proof}

The following Lemma will be used in the proof of Lemma \ref{lem:mv1-zero-slice}.
\begin{lemma}\label{lem:mv1-zero-slice}
Assume $s\ge 3$. There are no symmetric and disjoint elements $x,y \in \mathbb{N}[W]$ such that
\[
\ev_D(x)=\ev_D(y),\qquad \eps(x)=q-1,\qquad \eps(y)=2q-3.
\]

\end{lemma}

\begin{proof}
Assume there exists such a pair $x, y \in \mathbb{N}[W]$ that satisfies the conditions. Using the $\Psi$-identification, we write $x=\sum\limits_{i\in \Z/q\Z}[i]\otimes x_i$ and 
$y=\sum\limits_{i\in \Z/q\Z}[i]\otimes y_i$ where $x_i,y_i\in \N[\Z/R\Z]$ are disjoint. By Lemma~\ref{lem:F_v}(2), there exists $T\in \Q(\zeta_R)\cap\mathbb{R}$ such that
\begin{equation}\label{eq:ev=T}
 \ev_R(x_i)-\ev_R(y_i)=T
\qquad(\forall i\in \Z/q\Z).   
\end{equation}
Let $E:=\{\,i\in (\Z/q\Z)^\times:\ x_i=0\,\}$ and $e:=|E|$.

\smallskip
\noindent
\textbf{Step 1: show that $E\ne \emptyset$.} Assume that $E=\emptyset$, then we have $x_0=0$ and $x_i=[r_i]\in \N[Z/R\Z]$ for any $i\ne 0$. We first show that $\eps(y_i)\ne 1$ when $i\ne 0$. Otherwise, suppose $y_i=[s_i]$ for some $i\ne 0$. By Lemma \ref{lem:v-slice} applied to $X=[r_i]$ and $Y=[s_i]$, we obtain a contradiction. Therefore for any $i\ne 0$, we have $\eps(y_i)=0$ or $\eps(y_i)\ge2$. Also $\eps(y_0)\ge 1$ since $\eps(y)$ is odd and $y$ is symmetric. Note that
\begin{equation}\label{eq:sum eps(yi)}
    \sum_{i\neq 0}\eps(y_i)= \eps(y)-\eps(y_0)\le 2q-4,
\end{equation}
hence there exists some $i\neq 0$ with $y_i=0$. For such an $i$, by \eqref{eq:ev=T}, we have $x_i=[0]$. Now fix any $j\neq 0$. We have,
\[
\ev_R(y_j+[0])=\ev_R(x_j+y_i)=\ev_R(x_j),\qquad \eps(y_j+[0])\le (q-1)+1<p_3,\qquad \eps(x_j)=1.
\]By Proposition~\ref{prop:rigid nonsym}, we have $y_j+[0]=x_j=[r_j]$. Therefore $y_j=0$ and $x_j=[0]$ for all $j\ne 0$. So $x=\sigma((\Z/q\Z)^\times)\otimes [0]$ and $y=[0]\otimes y_0$. Lemma \ref{lem:F_v}(2) gives $\ev_D(x)=\sum\limits_{i\neq 0}\zeta_p^i=-1$ and $\ev_D(y)=\ev_R(y_0)$. Hence
\[
\ev_R(y_0+[0])=-1+1=0,\qquad \eps(y_0+[0])=(2q-3)+1=2q-2<2p_3.
\] But there is no vanishing sum of $R$-th roots of unity of weight $2q-2$ by Theorem \ref{thm:LamLeung}(3), a contradiction. Hence $E\ne \emptyset$.

\smallskip
\noindent
\textbf{Step 2: show that $E=(\Z/q\Z)^\times$.} 
Note that $x$ is symmetric and $E\ne \emptyset$, we have $q-1\ge e\ge 2$. Choose $i\in E$ such that $d:=\eps(y_i)$ is minimal among all $j\in E$. Since $e\,d\le \sum\limits_{j\in E}\eps(y_j)\le \sum\limits_{j\neq 0}\eps(y_j)\le 2q-4$, we have $d\le \frac{2q-4}{e}$.  Let $f:=|\{\,i\in \Z/q\Z:\ x_i=0\,\}|$. Note that $f\le e+1$.  

Suppose for contradiction that $E\ne (\Z/q\Z)^\times$. We choose a $k\in (\Z/q\Z)^\times\setminus E$. Then we have 
\begin{align*}
    \eps(x_k+y_i)& = \eps(y_i)+\eps(x_k)= d+\eps(x_k)\\  
    &=d+ \eps(x)-\sum_{\substack{\eps(x_i)\ne0 \\ i\ne k}}\eps(x_i)\le d+(q-1)-(q-f-1)=d+f\le \frac{2q-4}{e}+e+1.
\end{align*}
 Note that $\frac{2q-4}{e}+e+1<q+2\le p_3$. Therefore we have
\[
\ev_R(x_k+y_i)=\ev_R(y_k+x_i)=\ev_R(y_k),\qquad\eps(x_k+y_i)<p_3,\qquad\eps(y_k)\le q-2<p_3.
\]
Thus
Proposition~\ref{prop:rigid nonsym} implies $x_k+y_i=y_k$. Since $x_k$ and $y_k$ are disjoint, this implies $x_k=0$, which contradicts to the choice of $k$. Hence $E=(\Z/q\Z)^\times$.

\smallskip
\noindent
\textbf{Step 3: derive a contradiction.} 
By Step 2, we have $x=\sum\limits_{i\in \Z/q\Z}[i]\otimes x_i=[0]\otimes x_0$. Also \eqref{eq:ev=T} gives
\[
\ev_R(y_i)=-T,\qquad \eps(y_i)<p_3\qquad (\forall i\ne 0).
\]
Then by Proposition~\ref{prop:rigid nonsym}, $y_i=y_j$ for all $i\ne0 ,j\ne0$ . Hence \eqref{eq:sum eps(yi)} implies $\eps(y_i)\le 1$ for all $i\ne 0$. Since $\ev_R(y_i)=-T\in\mathbb R$, we have $y_i$ is either $0$ or $[0]$. 

If $y_i=[0]$ for every $i\ne 0$, then we have 
\[
\ev_R(x_0+[0])=\ev_R(y_0),\qquad \eps(x_0+[0])=\eps(x)+1=q<p_3,\qquad \eps(y_0)=q-2<p_3,
\]so Proposition~\ref{prop:rigid nonsym}
implies $x_0+[0]=y_0$, impossible by comparing weights.

Now we assume that $y_i=0$ for every $i\ne 0$. Then $y=[0]\otimes y_0$. Therefore
\[
\ev_R(x_0)=\ev_R(y_0),\qquad\eps(x_0)=q-1<p_3,\qquad \eps(y_0)=2q-3.
\]
If $s=3$, then $R$ is prime; thus, Lemma \ref{lem:one prime} yields a contradiction. If $s\ge 4$, then Theorem \ref{thm:lam-leung-comparison} implies $y_0\ge x_0$ since $(p_3-\eps(x_0))(p_4-1)\ge 3(p_4-1)>2q>\eps(y_0)$. Now
 consider $y_0-x_0\in \N[\Z/R\Z]$ which is a vanishing sum of weight $q-2<p_3$. Hence we get a contradiction by Theorem \ref{thm:LamLeung}(3). Therefore such a pair $x,y$ cannot
exist.
\end{proof}

Now let $A,B\in \N[U]$ be symmetric and disjoint with $\ev_M(A)=\ev_M(B)$ and $0<k:=\eps(A)=\eps(B)\le 2p$. Assume that $s\ge 2$. Using the $\Psi$-identification, write
\[
A=\sum_{u\in V}[u]\otimes A_u,
\qquad
B=\sum_{u\in V}[u]\otimes B_u,
\]
with $A_u,B_u\in \N[W]$ disjoint. By Lemma \ref{lem:F_v}(2), there exists $S\in\Q(\zeta_D)\cap \R$, such that
\begin{equation}\label{eq:A_v-B_v S}
    \ev_D(A_v)-\ev_{D}(B_v)=S,\qquad \forall v\in V.
\end{equation}
For $u\in V$, put $a_u:=\eps(A_u)$, $b_u:=\eps(B_u)$ and $m_u:=\eps(A_u)+\eps(B_u)$, and choose a
$v\in V^\times$ such that $m_v=\min\{m_u:u\in V^\times\}$. Since $A,B$ are disjoint and symmetric, $k$ is even. And we have $\sum\limits_{u\in V}a_u=\sum\limits_{u\in V}b_u=k\le2p$ and $\sum\limits_{u\in V}m_u=2k\le 4p$.

\begin{lemma}\label{lem:u-slice}
Assume that $1\le m_v\le 3$. After interchanging $A$ and $B$ if
necessary,  $(A_v,B_v)$ is one of the four pairs listed in
Lemma~\ref{lem:v-slice}. Then the following assertions hold:

(1) If $(a_v, b_v) \in \{(1, 0), (2, 1)\}$, then $(A_u,B_u)=(A_v,B_v)$ for all $u\in V^\times$.

(2) If $(A_u,B_u)\ne (A_v,B_v)$ for some $u\in V^\times$, then $q=p+2$ and 
\[
(a_u,b_u)= \begin{cases}
    (0,p), & \text{if }(a_v,b_v)=(2,0) \\
   
    (0,p-1)\text{ or }(1,p),   & \text{if }(a_v,b_v)=(3,0).
\end{cases}
  \] 

\end{lemma}

\begin{proof}
For any $u\in V^\times$, we have $\ev_D(A_u+B_v)=\ev_D(A_v+B_u)$ by \eqref{eq:A_v-B_v S}. We claim that $z_u:=(A_v+B_u)-(A_u+B_v)\in \Z[W]$ is $0$ or a vanishing sum of weight $q$ . If $s=2$, Lemma~\ref{lem:one prime} implies $z_u=n_u\cdot\sigma(W)$ for some $n_u\in\Z$. Since $\eps(A_u+B_v)=a_u+b_v\le p+1<q$, we cannot have $n_u<0$; otherwise, $A_u+B_v \ge \sigma(W)$. Hence $\eps(A_u+B_v)\ge \eps(\sigma(W))=q$, which is a contradiction. Hence $z_u\in \N[W]$. Now assume $s\ge 3$. By Theorem~\ref{thm:lam-leung-comparison} on $W$, either $A_v+B_u\ge A_u+B_v$, or $\eps(A_v+B_u)\ge (q-\eps(A_u+B_v))(p_3-1)$. It is easy to check that the second alternative is impossible. Therefore $z_u\in \N[W]$. In particular, we have $A_u\le A_v$. Since $\eps(z_u)= (a_v+b_u)-(a_u+b_v)\le 3+p\le q+1$, Theorem~\ref{thm:LamLeung}(3) implies $\eps(z_u)\in\{0,q\}$.

(1) If $(a_v,b_v)=(1,0)$ or $(2,1)$, then $\eps(z_u)=b_u+1-a_u\le p+1<q$, this implies $z_u=0$. Then $b_u+a_v=a_u+b_v$. Since $A_u\le A_v$, we have $a_u\le a_v$, hence $b_u\le b_v$. Therefore $a_u+b_u\le a_v+b_v=m_v$. By the minimality of $m_v$, equality must hold. So $a_u=a_v$ and $b_u=b_v$. Since $A_u\le A_v$, this implies $A_u=A_v$, and then $z_u=0$ implies $B_u=F$. 

(2) If $(A_u,B_u)\ne (A_v,B_v)$ for some $u\in V^\times$, then $z_u\ne 0$. Hence $\eps(z_u)=q$. Then by (1), we have $(a_v,b_v)=(2,0)$ or $(3,0)$. If $(a_v,b_v)=(2,0)$, then $q=b_u-a_u+(a_v-b_u)\le p+2$. This implies $q=p+2$, and then $a_u\le a_v\le 2$, $b_u\le p$ give $(a_u, b_u)=(0,p)$. If $(a_v,b_v)=(3,0)$, then $q=b_u-a_u+(a_v-b_u)\le p+3$. Again this implies $q=p+2$, and  $(a_u, b_u)=(0,p-1)\ \text{or}\ (1,p)$.
\end{proof}

\begin{lemma}\label{lem:m_v>3}
   If $\ev_D(A_v)-\ev_{D}(B_v)=S\ne 0$, then $m_v\ge 4.$ 
\end{lemma}

\begin{proof}
Suppose for contradiction that $m_v\le 3$. Since $S\ne 0$, we have $1\le m_v\le 3$. After interchanging $A$ and $B$ if necessary , $(A_v,B_v)$ belongs to the four pairs in Lemma~\ref{lem:v-slice}: 

\medskip
\noindent
\textbf{Case 1: $(A_v,B_v)=([0],0)$.}
By Lemma \ref{lem:u-slice}(1), we have $(A_u,B_u)=([0],0)$ for every $u\in V^\times$. Then by \eqref{eq:A_v-B_v S},
\[
\ev_D(A_0)=\ev_D(B_0+A_v),\quad
\eps(A_0)=k-(p-1)\le p+1<q,\quad
\eps(B_0+A_v)=k+1.
\]
If $s=2$, Lemma~\ref{lem:one prime} gives a contradiction. Now assume $s\ge 3$. If either $k\le 2p-2$ or
$q\ge p+4$, then
\[
(q-\eps(A_0))(p_3-1)=(q+p-1-k)(p_3-1)>k+1=\eps(B_0+A_v),
\]
so Theorem~\ref{thm:lam-leung-comparison} implies $B_0+A_v\ge A_0$. Then
\[
z_0:=(B_0+A_v)-A_0\in \N[W],\qquad \ev_D(z_0)=0,\qquad \eps(z_0)=p<q,
\]
contradicting Theorem \ref{thm:LamLeung}(3). Thus the only remaining possibility is $k=2p$ and
$q=p+2$. If $B_0+A_v\ge A_0$, the same argument again gives a vanishing
sum in $\N[W]$ of weight $p$, impossible. Hence $B_0+A_v\not\ge A_0$. Applying
Theorem~\ref{thm:lam-leung-comparison} first with $\nu=\eps_0$ and then with
$\nu=\eps$, we get $\eps_0(A_0)=\eps(A_0)=q-1$. Since $A_0$ is symmetric and $\eps_0(A_0)$ is even, $[0]$ is not in the support
of $A_0$. Therefore $x_0:=A_0\in \N[W]$ and $y_0:=B_0+A_v\in \N[W]$ are symmetric and disjoint, with
\[
\ev_D(x_0)=\ev_D(y_0),\qquad
\eps(x_0)=q-1,\qquad
\eps(y_0)=2q-3.
\]
This contradicts Lemma~\ref{lem:mv1-zero-slice}. Hence this case is
impossible.

\medskip
\noindent
\textbf{Case 2: $(A_v,B_v)=(C,[0])$.}
By Lemma \ref{lem:u-slice}(1), we have $(A_u,B_u)=(C,[0])$ for every $u\in V^\times$. Then $\ev_D(A_0+B_v)=\ev_D(B_0+A_v)$,  
$\eps(A_0+B_v)=k-2(p-1)+1\le 3$ and $\eps(B_0+A_v)=k-p+3\le p+3$. If $s=2$, Lemma~\ref{lem:one prime} gives a contradiction. If $s\ge 3$, then Theorem~\ref{thm:lam-leung-comparison} implies $B_0+A_v\ge A_0+B_v$. Therefore
\[
z_0:=(B_0+A_v)-(A_0+B_v)\in \N[W],\qquad \ev_D(z_0)=0,\qquad \eps(z_0)=p<q,
\]
impossible by Theorem \ref{thm:LamLeung}(3). So this case cannot occur.

\medskip
\noindent
\textbf{Case 3: $(A_v,B_v)=(C,0)$.}
By Lemma \ref{lem:u-slice}, for every $u\ne 0$, either $(A_u,B_u)=(C,0)$, or, when $q=p+2$, of type $(a_u,b_u)=(0,p)$.

Suppose $(a_u,b_u)=(0,p)$ for some $u\ne 0$. Then by symmetry the same holds for $-u$. Since $B_u$ and $B_{-u}$ already contribute $2p$ to
$\eps(B)=k\le 2p$, we get $k=2p$ and $B_w=0$ for all $w\ne \pm u$. Since $(a_v,b_v)=(2,0)$, we have $u\ne \pm v$, so $p\ge 5$. Therefore
$(A_w,B_w)=(C,0)$ for all $w\ne 0,\pm u$. Hence
\[
\ev_D(A_0)=\ev_D(C),\qquad \eps(A_0)=6<q,\qquad \eps(C)=2<q.
\]
So Proposition~\ref{prop:rigid nonsym} 
implies $A_0=C$, impossible by comparing weights. 

Therefore $(A_u,B_u)=(C,0)$ for all $u\in V^\times$. Now
\[
\ev_D(A_0)=\ev_D(B_0+A_v),\quad
\eps(A_0)=k-2(p-1)\le 2,\quad
\eps(B_0+A_v)=k+2\le 2p+2.
\]
If $s=2$, Lemma~\ref{lem:one prime} gives a contradiction. If $s\ge 3$, then Theorem~\ref{thm:lam-leung-comparison} implies $B_0+A_v\ge A_0$. Therefore
\[
z_0:=(B_0+A_v)-A_0\in \N[W],\qquad \ev_D(z_0)=0,\qquad \eps(z_0)=2p,
\]
which is impossible by Theorem~\ref{thm:LamLeung}(3), since $2p<2q$. Hence
this case is impossible.

\medskip
\noindent
\textbf{Case 4: $(A_v,B_v)=(P,0)$.}
By Lemma \ref{lem:u-slice}, for every $u\ne 0$,  $(a_u,b_u)\in\{(3,0),(0,p-1),(1,p)\}$
and the latter two can occur only when $q=p+2$. Let $\alpha,\beta,\gamma$ be the numbers of these three types. Then
$\alpha,\beta,\gamma$ are even, $\alpha\ge 2$, and
\[
\alpha+\beta+\gamma=p-1,\qquad
3\alpha+\gamma\le 2p,\qquad
(p-1)\beta+p\gamma\le 2p.
\]
It follows that $(p,\alpha,\beta,\gamma) \in \{(3,2,0,0), (5,2,2,0), (5,2,0,2)\}$.

When $(p,\alpha,\beta,\gamma)=(3,2,0,0)$, the zero-slice equation yields $\ev_D(B_0+P)=0$. However, $\eps(B_0+P)=9$, which is impossible by Theorem \ref{thm:LamLeung}(3). In the case $(p,\alpha,\beta,\gamma)=(5,2,2,0)$, we have $q=7$ and the zero-slice equation gives $\ev_D(A_0)=\ev_D(B_0+P)$. Since $\eps(A_0)=k-6\le 4<7$ and $\eps(B_0+P)=k-5\le 5<7$,  Proposition~\ref{prop:rigid nonsym} implies $A_0=B_0+P$, a contradiction by comparing weights. Finally, if $(p,\alpha,\beta,\gamma)=(5,2,0,2)$, then $q=7$ and $k=10$, and the zero-slice equation becomes $\ev_D(A_0)=\ev_D(P)$. As $\eps(A_0)=2<7$ and $\eps(P)=3<7$, Proposition~\ref{prop:rigid nonsym} implies $A_0=P$, which is again impossible by comparing weights.

Thus we have $m_v\ge 4$.     
\end{proof}

Now, we are ready to show the Fourier rigidity in the square-free case.
\begin{proposition}\label{prop:rigid symm dis}
Let $A,B\in \N[U]$ be symmetric and disjoint. Assume $0<\eps(A)=\eps(B)\le 2p$ and 
$\ev_M(A)=\ev_M(B)$.
Then $\eps(A)=\eps(B)=2p$ and  $\ev_M(A)=\ev_M(B)=0$.
\end{proposition}

\begin{proof}
  We argue by induction on the number $s$ of prime factors of
$M$. The case $s=1$ follows from Lemma \ref{lem:one prime}. Assume $s\ge 2$,  and let $S$ be the constant from Lemma \ref{lem:F_v}(2): $\ev_D(A_u)-\ev_D(B_u)=S$. Recall that $a_u:=\eps(A_u)$, $b_u:=\eps(B_u)$ and $m_u:=\eps(A_u)+\eps(B_u)$.

If $S=0$, then for every $u\neq 0$ we have $\ev_D(A_u)=\ev_D(B_u)$ and 
$a_u,b_u\le p<q$.
Proposition~\ref{prop:rigid nonsym}
and the disjointness implies $A_u=B_u=0$ for every $u\neq 0$. Thus $A=[0]\otimes A_0$ and $B=[0]\otimes B_0$.
Now we have $\ev_D(A_0)=\ev_D(B_0)$ and $\eps(A_0)=\eps(B_0)\le 2p<2q$. If $(A_0,B_0)\ne (0,0)$,  the
induction hypothesis implies $\eps(A_0)=\eps(B_0)=2q$, a contradiction. Thus $A=B=0$, a contradiction. Hence $S\ne 0$.

Since $S\ne 0$, we have $m_u \ge 1$ for all $u \in V$. 
By Lemma \ref{lem:F_v}(1), we have $a_u=a_{-u}$ and $b_u=b_{-u}$. Hence $m_0=2k-\sum\limits_{u\ne 0}m_u$ is even and $m_0\ge 2$. Choose
a $v\in V^\times$ such that $m_v=\min\{m_u:u\in V^\times\}$. By Lemma \ref{lem:m_v>3}, we have $m_v\ge 4$.

We claim that $m_v=4$. Suppose for contradiction that $m_u\ge 5$ for all $u\ne 0$. If $p \ge 5$, then $\sum\limits_{u \in V} m_u \ge 5(p-1)+2 > 4p$, a contradiction. If $p=3$, then we must have $m_0 = 2$ and $m_1 = m_{-1} = 5$. Since $a_1,b_1 \le p = 3$, up to interchanging $A$ and $B$, we may assume $(a_1, b_1) = (3, 2)$, which implies $(a_0,b_0)=(0,2)$. Then $\ev_D(A_1 + B_0) = \ev_D(B_1)$. The two sides have weights $5$ and $2$, both at most $q$.  Proposition~\ref{prop:rigid nonsym}  implies $A_1 + B_0 = B_1$, a contradiction.

  By the minimality of $m_v$, we have $2k=\sum\limits_{u\in V}m_u\ge m_0+4(p-1)\ge 4p-2$. As $k\le 2p$ is even, we have $k=2p$. There are only two possibilities:

\smallskip
\noindent
{\it (a) $m_0=4$, then $m_u=4$ for every $u\in V$;}

\smallskip
\noindent
{\it (b) $m_0=2$, then $\exists u_0\in V^\times$ s.t. $m_{u_0}=m_{-u_0}=5$ and $m_u=4$ for all $u\ne0,\pm u_0$.}

\smallskip
\noindent
We first rule out {\it (b)}. If $p=3$, then this is impossible because there is
only one off-zero pair. So $p\ge 5$, hence $q\ge 7$. Then $\ev_D(A_0+B_v)=\ev_D(B_0+A_v)$, and both sides have weight at most $6<q$. Therefore Proposition~\ref{prop:rigid nonsym} implies $A_0+B_v=B_0+A_v$, hence $A_0-B_0=A_v-B_v$. Since $A_0,B_0$ and $A_v,B_v$ are disjoint, equality of differences implies
equality of positive and negative parts. Thus $A_0=A_v$ and $B_0=B_v$, contradicting $m_0=2$ and $m_v=4$. So {\it (b)} cannot occur.

\smallskip
\noindent
We are therefore in case {\it (a)}: every slice has weight $4$. Fix $u\neq 0$. Comparing the slices $u$ and $-u$, we get $\ev_D(A_u+B_{-u})=\ev_D(B_u+A_{-u})$, and both sides have weight $4<q$. So Proposition~\ref{prop:rigid nonsym} implies $A_u-B_u=A_{-u}-B_{-u}$. Hence $D_u:=A_u-B_u$ is symmetric. Since $A_u,B_u$ are disjoint and $m_u=4$, $(a_u,b_u)=(4,0),(2,2)$ or $(0,4)$ for all $u\in V$. Let $a,b,c$ be the numbers of  these three types. Then $4a+2b=2b+4c=2p$, So $a=c$. If $a=c>0$, then there exist $u,w\in V$ such that $(a_u,b_u)=(4,0)$ and $(a_w,b_w)=(0,4)$, then $\ev_D(A_u+B_w)=\ev_D(B_u+A_w)=0$, while $A_u+B_w\in \N[W]$ is vanishing sum of weight $8$, impossible
because $8<2q$ and $8$ is not a positive vanishing weight on $W$ by Theorem \ref{thm:LamLeung}(3).

Thus, every slice has type $(2,2)$. Fixing $u\ne 0$ and comparing the $0$-slice with the $u$-slice, Proposition~\ref{prop:rigid nonsym} implies $A_0+B_u=B_0+A_u$, and hence $A_u-B_u=A_0-B_0$. Since each pair $(A_u,B_u)$ is disjoint, it follows that $A_u=A_0$ and $B_u=B_0$ for all $u\in V$. Consequently, $A=\sigma(V)\otimes A_0$ and $B=\sigma(V)\otimes B_0$, which yields $\eps(A)=\eps(B)=2p$ and $\ev_M(A)=\ev_M(B)=0$.
\end{proof}

\begin{proposition}\label{prop:rigid symm}
Let $X,Y\in \N[U]$ be symmetric. Assume that $\eps(X)=\eps(Y)\le 2p$ and  $\ev_M(X)=\ev_M(Y)\neq 0$. Then $X=Y$ in $\N[U]$.
\end{proposition}

\begin{proof}
Let $D:=X-Y=D^+-D^-$. Since $D$ is symmetric,  $D^+$ and $D^-$ are symmetric and disjoint. We have  $\eps(D^+)=\eps(D^-)\le 2p$ and 
$\ev_M(D^+)=\ev_M(D^-)$.
 If $D^+=D^-=0$,
then $X=Y$. Otherwise Proposition~\ref{prop:rigid symm dis} implies $\eps(D^+)=\eps(D^-)=2p$ and 
$\ev_M(D^+)=\ev_M(D^-)=0$. Since $D^+\le X$ and $\eps(X)\le 2p$, this implies $X=D^+$. Similarly $Y=D^-$. Hence $\ev_M(X)=0$, a contradiction. Therefore $D=0$, and so $X=Y$.
\end{proof}

\begin{remark}
    Propositions \ref{prop:rigid symm dis} and \ref{prop:rigid symm} are equivalent. Indeed, if Proposition \ref{prop:rigid symm} holds, the disjointness of $A$ and $B$ implies $\ev_M(A)=\ev_M(B)=0$. Then Theorem \ref{thm:LamLeung}(3) then  $\varepsilon(A)=\varepsilon(B)$ to $\{p, 2p\}$. Since the weight must be even by disjointness, we have $\varepsilon(A)=\varepsilon(B)=2p$.
\end{remark}

\begin{example}\label{ex:rigid-symm ex}
    (1) The oddness assumption is essential in Proposition \ref{prop:rigid symm}. Take 
$M=6,U=\Z/6\Z$. Let $X=2[0]+[2]+[4]$ and $Y=[0]+[3]+[1]+[5]$. Both are symmetric and $\eps(X)=\eps(Y)=4=2p$. We have $\ev_6(X)=\ev_{6}(Y)=1\ne 0$ but $X\ne Y$.

(2) The condition that the common evaluation is nonzero is also necessary. Let $M=pq$ with $2<p<q$ primes and $U=\Z/M \Z$. Let $C_p$ be the subgroup of $U$ of order $p$. Take $X=\sigma(C_p)+\sigma(C_p)$ and $Y=\sigma(1+C_p)+\sigma(-1+C_p)$. Both are symmetric and  $\eps(X)=\eps(Y)=2p$. We have $\ev_{M}(X)=\ev_M(Y)=0$ but $X\ne Y$.

(3) The condition that $\eps(X)=\eps(Y)$ is necessary. Let $M=pq$ with $2<p<q$ primes and $U=\Z/M \Z$.  Take $X=[0]+\sigma(C_p)$ and $Y=[0]$. Both are symmetric and  $\eps(X)=1+p<2p,\eps(Y)=1<2p$. We have $\ev_{M}(X)=\ev_M(Y)=1\ne 0$ but $X\ne Y$.

(4) The bound $2p$ is optimal. Let $M=pq$ with $2<p<q$ primes and $U=\Z/M \Z$. Take $X=[0]+\sigma(C_p)+\sigma(C_p)$ and $Y=[0]+\sigma(1+C_p)+\sigma(-1+C_p)$. Both are symmetric and  $\eps(X)=\eps(Y)=2p+1$. We have $\ev_{M}(X)=\ev_M(Y)=1\ne 0$ but $X\ne Y$.
\end{example}

\subsection{Lifting rigidity to odd \texorpdfstring{$N$}{N}}\label{sec:general odd}
In this subsection, we first establish a coset decomposition for general $N$, then apply it to prove Theorem \ref{thm:rig}.

For an integer $N=\prod\limits_{i=1}^{s} p_i^{r_i}$ with distinct  primes $p_1<\cdots<p_s$ and $r_i\ge 1$, let its radical be $M:=\prod\limits_{i=1}^{s} p_i$ and denote $Q:=\frac{N}{M}$. 
Let $G:=\Z/N\Z$,
$U:=\Z/M\Z$ and 
\[
H:=\{Qu\pmod N:\ u\in U\}\le G.
\]
Then $|H|=M$ and $G=\bigsqcup\limits_{r=0}^{Q-1}(r+H)$. For each $0\le r\le Q-1$, we define the $\Z$-linear embedding:
\[
\iota_r: \mathbb{Z}[U] \hookrightarrow \mathbb{Z}[G],\quad \sum_{u\in U} a(u)[u]\mapsto \sum_{u\in U}a(u)[r+Qu].
\]
Then $\Z[G]=\bigoplus\limits_{r\in G/H}\operatorname{Im}(\iota_r)$. Therefore after choosing the representative set $\{0,1,2,\dots,Q-1\}$ of $G/H$, any $F\in \Z[G]$ can be written uniquely as $F=\sum\limits_{r=0}^{Q-1}\iota_r(F_r)$ with $F_r\in \Z[U]$.

Parts (1) and (2) of the following Lemma are reformulations of results from \cite[Theorem 2.2]{lenstra1978vanishing} and \cite[Theorem 3.1]{lam2000vanishing}, which provide a general technique for reducing problems on vanishing sum of $N$-th roots of unity to those on $M$-th roots of unity. For the sake of completeness, we give the detailed proof below.

\begin{lemma}\label{lem:F_r} (1) We have 
\vspace{-1.5em}
\[
\operatorname{ev}_{N}(F)=\sum\limits_{r=0}^{Q-1}\zeta_N^r\,\operatorname{ev}_{M}(F_r).
\]

(2) Let $A,B\in \Z[G]$. If $\ev_{N}(A)=\ev_{N}(B)$, then $\ev_M(A_r)=\ev_M(B_r)$ for all $r$.

(3) If $F\in \Z[G]$ is symmetric.  Then $F_0$ is symmetric in $\Z[U]$. Moreover, for every $1\le r\le Q-1$, we have $\eps(F_{Q-r})=\eps(F_r)$ and 
$\ev_M(F_{Q-r})=\zeta_M^{-1}\overline{\operatorname{ev}_{M}(F_r)}$.
\end{lemma}

\begin{proof}
(1) Write $F=\sum\limits_{x\in G}a(x)[x]$ with $a(x)\in\Z$. Then $F_r=\sum\limits_{u\in U}c_{r}(u)[u]\in \Z[U]$ where $c_r(u):=a(r+Qu)$.
 Since $\operatorname{ev}_N([r+Qu])=\zeta_N^{r+Qu}=\zeta_N^r\zeta_M^u$, 
we have  $\operatorname{ev}_N(\iota_r(F_{r}))
=\zeta_N^r\,\operatorname{ev}_{M}(F_{r})
=\zeta_N^r\,\operatorname{ev}_{M}(F_r)$. Hence $\operatorname{ev}_{N}(F)=\sum\limits_{r=0}^{Q-1}\zeta_N^r\,\operatorname{ev}_{M}(F_r)$.

(2) Subtracting the expansions in (1) gives
$0=\ev_{N}(A)-\ev_{N}(B)=\sum\limits_{r=0}^{Q-1}\zeta_N^r\,(\ev_M(A_r)-\ev_M(B_r))$, where $\ev_M(A_r)-\ev_M(B_r)\in \Q(\zeta_M)$.
Since $\zeta_M=\zeta_N^Q$, we have $\Q(\zeta_M)\subset \Q(\zeta_N)$ and $[\Q(\zeta_N):\Q(\zeta_M)]=\frac{[\Q(\zeta_N):\Q]}{[\Q(\zeta_M):\Q]}=\frac{\varphi(N)}{\varphi(M)}=Q$. Since $x^Q-\zeta_M$ is the minimal polynomial of $\zeta_N$ over $\Q(\zeta_M)$, the set
$\{1,\zeta_N,\dots,\zeta_N^{Q-1}\}$ is a $\Q(\zeta_M)$-basis of $\Q(\zeta_N)$.
Hence $\ev_M(A_r)=\ev_M(B_r)$ for all $r$.

(3) Write $F$ and $F_r$ as in (1). Since $F$ is symmetric, we have $a(x)=a(-x)$ for all $x\in G$. For $r=0$, we get $c_0(-u)=a(-Qu)=a(Qu)=c_0(u)$
for all $u\in U$, so $F_0$ is symmetric in $\Z[U]$.

Now fix $r$ with $1\le r\le Q-1$. For every $u\in U$, since $-(r+Qu)\equiv(Q-r)+Q(-u-1)\pmod N$, we have
\[
c_r(u)=a(r+Qu)=a(-(r+Qu))=a\bigl((Q-r)+Q(-u-1)\bigr)=c_{Q-r}(-u-1).
\]
Replacing $u$ by $-u-1$, we obtain $c_{Q-r}(u)=c_r(-u-1)$, and hence
\[
F_{Q-r}
=\sum_{u\in U} c_{Q-r}(u)[u]
=\sum_{u\in U} c_r(-u-1)[u]
=\sum_{u\in U} c_r(u)[-u-1].
\]

Since the map $u\mapsto -u-1$ is a bijection of $U$, it follows that $\eps(F_{Q-r})=\sum\limits_{u\in U} c_r(u)=\eps(F_r)$. Also, $\ev_M(F_{Q-r})
=\sum\limits_{u\in U} c_r(u)\zeta_M^{-u-1}
=\zeta_M^{-1}\sum\limits_{u\in U} c_r(u)\zeta_M^{-u}=\zeta_M^{-1}\overline{\operatorname{ev}_{M}(F_r)}$.
\end{proof}

Now we are ready to prove the main  Theorem:

\begin{proof}[Proof of Theorem \ref{thm:rig}]
We can write $X=\sum\limits_{r=0}^{Q-1}\iota_r(X_r)$ and 
$Y=\sum\limits_{r=0}^{Q-1}\iota_r(Y_r)$ where $X_r,Y_r\in \N[U]$. Furthermore, Lemma~\ref{lem:F_r}(2) implies that $\ev_M(X_r)=\ev_M(Y_r)$ for all $ 0\le r\le Q-1$. 

(1) Since $\sum\limits_{r=0}^{Q-1}\eps(X_r)=\eps(X)\le p$ and 
$\sum\limits_{r=0}^{Q-1}\eps(Y_r)=\eps(Y)\le p$, 
we have $\eps(X_r)\le p$ and 
$\eps(Y_r)\le p$ for all $0\le r\le Q-1$.  If $\ev_M(X_r)=\ev_M(Y_r)\neq 0$, then Proposition \ref{prop:rigid nonsym} implies $X_r=Y_r$. If $\ev_M(X_r)=\ev_M(Y_r)= 0$, we  show that $X_r=Y_r=0$. Suppose for contradiction that $X_r\ne 0$, then Theorem \ref{thm:LamLeung}(3) implies $\eps(X_r)=p$. Therefore $X_s=0$ for all $s\neq r$. Hence by Lemma \ref{lem:F_r}(1), $\ev_N(X)=\zeta_N^r\ev_M(X_r)=0$, a contradiction. So $X_r=0$. Similarly $Y_r=0$.  Thus $X_r=Y_r$ for every $r$, and therefore $X=Y$.

(2) Fix an $r$ with $1\le r\le Q-1$. Since $X$ and $Y$ are symmetric, by Lemma \ref{lem:F_r}(3), we have $\eps(X_r)=\eps(X_{Q-r})\le p$ and 
$\eps(Y_r)=\eps(Y_{Q-r})\le p$. If $\ev_M(X_r)=\ev_M(Y_r)\neq 0$, then Proposition \ref{prop:rigid nonsym} implies
$X_r=Y_r$. If $\ev_M(X_r)=\ev_M(Y_r)= 0$, we  show that $X_r=Y_r=0$. Suppose for contradiction that $X_r\ne 0$. Theorem \ref{thm:LamLeung}(3) then implies $\varepsilon(X_r)\ge p$. Since $2p = \varepsilon(X_r)+\varepsilon(X_{Q-r}) \le \varepsilon(X) \le 2p$, it follows that $\varepsilon(X)=2p$ and $X_u=0$ for all $u\ne r, Q-r$. By Lemma \ref{lem:F_r}, we obtain $\text{ev}_{N}(X) = \zeta_N^r\ev_M(X_r) + \zeta_N^{Q-r}\ev_M(X_{Q-r}) = 0$,  a contradiction. Hence we  have $X_r=0$. Similarly $Y_r=0$.

So we have $X_r=Y_r$ for all $r\ne 0$. Therefore we have $\eps(X_0)=\eps(Y_0)\le 2p$ and $X_0,Y_0$ are symmetric. If $\ev_M(X_0)=\ev_M(Y_0)\ne0$, then Proposition \ref{prop:rigid symm} implies $X_0=Y_0$.  Suppose $\operatorname{ev}_M(X_0) = \operatorname{ev}_M(Y_0) = 0$, then Theorem \ref{thm:LamLeung}(3) implies $\eps(X_0)=\eps(Y_0)\in\{0, p, 2p\}$. Clearly, if $\eps(X_0) = \eps(Y_0) = 0$, then $X_0 = Y_0 = 0$. In the case where $\eps(X_0) = \eps(Y_0) = p$, the symmetry of $X_0$ and $Y_0$ along with the fact that $N$ is odd implies $X_0 = Y_0 = \sigma(C_p)$. Finally, if $\eps(X_0) = \eps(Y_0) = 2p$, then $X_r = Y_r = 0$ for all $r \neq 0$, which yields $\operatorname{ev}_N(X) = \operatorname{ev}_R(X_0) = 0$, a contradiction. Thus, $X_0 = Y_0$ holds in all cases. Therefore $X=Y$.
\end{proof}

As a consequence of Theorem \ref{thm:rig}, we have the following results on representation of sums of roots of unity and cosines:

\begin{cor}\label{cor:number of sol}
Let $N$ be odd, and let $p$ be the smallest prime divisor of $N$. 

(1) Let
$c\in \Q(\zeta_N)$ with $c\neq 0$. Then up to permutation, the equation
\[
\zeta_N^{x_1}+\cdots+\zeta_N^{x_k}=c
\]
has at most one solution for $1\le k\le p,\quad x_1,\cdots,x_k\in\Z/N\Z$.

(2) Let
$c\in \R$ with $c\neq 0$, and fix an integer $1\le k\le p$. Then up to permutation, the equation
\[
\cos\frac{2\pi x_1}{N}+\cdots+\cos\frac{2\pi x_k}{N}=c,
\]
has at most one solution for $x_1,\cdots,x_k\in[0,\frac{N-1}{2}]\cap\Z$.
\end{cor}

\section{Spectral applications to cyclic Cayley graphs}\label{sec:app cay}

In this section, we use the algebraic results of the previous sections to study the eigenvalues of cyclic Cayley graphs.

\subsection{Eigenvalues of cyclic Cayley graphs}\label{sec:eigenvalues}

Let $(G,S)$ be the Cayley graph associated with a finite abelian group $G$ and a symmetric generating set $S$. Recall that a symmetric generating set is a subset $S\subset G\setminus\{0\}$ that generates $G$ and satisfies $-s\in S$ for all $s\in S$. The discrete Laplace operator of the graph is the difference operator defined by
\[
(\Delta f)(x)=-\sum_{s\in S}\bigl(f(x+s)-f(x)\bigr).
\]
We are interested in the distribution of the eigenvalues of this Laplace operator.  Since the Cayley graph $(G,S)$ is $|S|$-regular, we have $\Delta = |S|I - A$, where $(Af)(x)=\sum\limits_{s\in S}f(x+s)$ is the adjacency operator. Therefore, they share the same eigenvectors, and their eigenvalues are determined by the relation $\lambda_\Delta=|S|-\mu_A$. Thus, investigating the eigenvalue distribution of the Laplace operator is equivalent to studying the eigenvalues of $A$.

The eigenvalues of the graph $(G,S)$ are defined as the eigenvalues of the adjacency operator $A$, and are given by
	\begin{equation}
		\mu_{\chi}=\sum_{g\in S}\chi(g), \label{eq:eigen cos}
	\end{equation}
	where $\chi$ is any character of $G$ since the characters of $G$ form a common
eigenbasis for the adjacency operator, see \cite{babai1979spectra}.
	
	In the special case where $G=\mathbb{Z}/N\mathbb{Z}$ is a cyclic group, the characters of $\mathbb{Z}/N\mathbb{Z}$ are given by $\chi_{k}(j)=e^{\frac{2\pi ikj}{N}}$ for $k,j\in \{0,1,\dots,N-1\}$. Hence, the eigenvalues of $(\mathbb{Z}/N\mathbb{Z}, S)$ are given by $\mu_{k}=\sum\limits_{j\in S} e^{i \frac{2k\pi j}{N}}$, where $k=0,1,\dots, N-1$. If $S$ is symmetric with $|S|$ even, we write $S=\{\pm s_1,\pm s_2,\dots,\pm s_t\}$. Then the eigenvalue $\mu_{k}=C_S(k)=\sum\limits_{j=1}^{t}2\cos{\frac{2k\pi s_j}{N}}$ is the structured cosine sum.

\begin{definition}\label{def:multiplicity}
For $\mu\in\R$, we define its multiplicity 
\[
\operatorname{mult}(\mu):=|\{k\in\{0,1,\dots,N-1\}:\ \mu_k=\mu\}|.
\]
\end{definition}

\subsection{The zero eigenvalue criteria}\label{sec:zero-eigenvalue}

Let $G=\Z/N\Z$ and let $S\subset G$ be a symmetric generating set. We identify a multiset $T$ in $G$ with its incidence element $\sigma(T)\in \N[G]$.
For every $k\in\{0,1,\dots,N-1\}$, $\mu_k=\sum\limits_{s\in S}\zN^{ks}=\ev_N(kS)$, where $kS$ is regarded as a symmetric multiset.  So the existence of the zero eigenvalue is equivalent to the existence of a symmetric vanishing multiset among the dilates $kS$.

For $N=p^aq^b$, Theorem~\ref{thm:symmetric-vanish-paqb} gives an explicit criterion.

\begin{theorem}\label{thm:zero-criterion-paqb}
Let $N=p^aq^b$ with distinct odd primes $p<q$, and let $S\subset \Z/N\Z$ be a symmetric generating set.
Then $0$ is an eigenvalue of $(\Z/N\Z,S)$ if and only if there exists $k\in\{1,\dots,N-1\}$ such that the multiset $kS$ is a finite sum of blocks of the types
\[
A\!\Big(u;\frac{N}{p}\Big),\qquad
A\!\Big(v;\frac{N}{q}\Big),\qquad
D_{p,q}(N).
\]
for some suitable $u,v\in G$.
\end{theorem}

\begin{proof}
Since $N$ is odd, the symmetric set $S$ has even cardinality. Hence, by Remark~\ref{rem:symmetrized-odd}, there exists $Y\in\N[G]$ such that $S=\sym(Y)$, and therefore $kS=\sym(kY)$ for every $k$. Here, we identify a multiset $T$ in $\Z/N\Z$ with its incidence element $\sigma(T)\in \N[\Z/NZ]$. Then the result follows from Theorem~\ref{thm:symmetric-vanish-paqb}.
\end{proof}

\begin{remark}
Let $N=2^aq^b$ with $q$ an odd prime.
Then Remark~\ref{rem:0 2q} gives the corresponding zero-eigenvalue criterion:
$0$ is an eigenvalue of $(\Z/N\Z,S)$ if and only if there exists $k\in\{1,\dots,N-1\}$ such that the multiset $kS$ is a finite sum of the explicit vanishing blocks $A(u;N/2)$, $A(v;N/q)$, $P_2$, $\frac{N}{4}+P_2$ when $4\mid N$, $C_q$, and $\frac{N}{2}+C_q$.
\end{remark}

For general $N$, the weight of the minimal vanishing sum of roots of unity is at least $p_1$. Hence for an odd $N$, $0$ is eigenvalue for $(\Z/N\Z,S)$ only when $|S|\ge 2p_1$. We have the following result.

\begin{theorem}\label{thm:zero-2p1}
Let $N=p_1^{r_1}\cdots p_s^{r_s}$ with $p_1<\cdots<p_s$, and let $S\subset \Z/N\Z$ be a symmetric generating set with $|S|=2p_1$.
Then $0$ is an eigenvalue of $(\Z/N\Z,S)$ if and only if there exist $k\in\{1,\dots,N-1\}$ and $u\in \Z/N\Z$ such that $kS=A\!(u;\frac{N}{p_1})$ as multisets.
\end{theorem}

\begin{proof}
Because $S\subset G\setminus\{0\}$ is symmetric and has even cardinality, every fixed point of $x\mapsto -x$ occurs in $S$ with even multiplicity.
Hence, by Remark~\ref{rem:symmetrized-odd}, there exists $Y\in\N[G]$ such that $S=\sym(Y)$, and therefore $kS=\sym(kY)$ for every $k$. Then Theorem~\ref{thm:weight-2p1} applies and gives exactly the stated criterion.
\end{proof}

\subsection{Multiplicities of rational eigenvalues}\label{sec:rational}

Let $N=p_1^{r_1}\cdots p_s^{r_s}$ with $p_1<p_2<\cdots<p_s$ distinct primes and $r_i$ positive integers. Let $(\mathbb{Z}/N\mathbb{Z},S)$ be the Cayley graph associated with the cyclic group $\mathbb{Z}/N\mathbb{Z}$ and a symmetric generating set $S$. In this subsection, we study rational eigenvalues of the Cayley graph $(\mathbb{Z}/N\mathbb{Z},S)$ and their multiplicities.

\begin{definition}\label{def:Ar}
Let $r\in\mathbb{Q}$. We define $\mathcal{A}_r(N,S):=\left\{1\le k\le N:\  \mu_k=r\right\}$. Then the multiplicity $\operatorname{mult}(r)=|\mathcal{A}_r(N,S)|$.

\end{definition}

Let $r \in \mathbb{Q}$ and $1 \le k \le N-1$. We write
\[
d=\gcd(k,N),\qquad M=\frac{N}{d},\qquad k=du.
\]
Let $\widebar{s}$ be the image of $s$ under the canonical quotient morphism $\mathbb{Z}/N\mathbb{Z}\to \mathbb{Z}/M\mathbb{Z}$. Since $\gcd(u,M)=1$, the morphism $\sigma_u:\mathbb{Q}(\zeta_M)\to \mathbb{Q}(\zeta_M)$ defined by $\sigma_u(\zeta_M)=\zeta_M^u$ is a Galois automorphism over $\Q$. 

\begin{proposition}\label{prop:gcd reduction}
 We have  $\mu_k=\sigma_u(\mu_d)$.
In particular, $k\in \mathcal{A}_r(N,S)$ if and only if $d\in \mathcal{A}_r(N,S)$, i.e., $\mu_k=r$\ if and only if $\mu_d=r$.

\end{proposition}

\begin{proof}
Since $k=du$, we have
\[
\mu_k
=
\sum_{s\in S} e^{2\pi i ks/N}
=
\sum_{s\in S} e^{2\pi i dus/N}
=
\sum_{s\in S} e^{2\pi i u\overline{s}/M}\in \Q(\zeta_M).
\]
On the other hand,
\[
\mu_d
=
\sum_{s\in S} e^{2\pi i ds/N}
=
\sum_{s\in S} e^{2\pi i \overline{s}/M}\in \Q(\zeta_M).
\]
Therefore $\mu_k=\sigma_u(\mu_d)$.  Because $r\in\mathbb{Q}$ is fixed by every Galois automorphism, we obtain $\mu_k=r\iff \mu_d=r$.
\end{proof}

\begin{cor}\label{cor:A general}
Let $r\in\mathbb{Q}$. If $\mathcal{A}_r(N,S)\ne \varnothing$ and $d_r$ is the smallest number in $\mathcal{A}_r(N,S)$, then $d_r\mid N$.
\end{cor}

\begin{proof}
Let $d=\gcd(d_r,N)$. By Proposition \ref{prop:gcd reduction}, we have $d\in \mathcal{A}_r(N,S)$. Since $d\le d_r$ and $d_r$ is the smallest element of $\mathcal{A}_r(N,S)$, we must have $d=d_r$. Therefore $d_r\mid N$.
\end{proof}

We have the following formula for multiplicities:

\begin{proposition}\label{prop:mr exact}
Let $r\in\mathbb{Q}$. Then
\[
\mathcal{A}_r(N,S)=\coprod_{\substack{d\mid N\\ \mu_d=r}}
\left\{1\le k\le N-1:\  \gcd(k,N)=d\right\}.
\]
Consequently,
\[
\operatorname{mult}(r)=\sum_{\substack{d\mid N\\ \mu_d=r}}\varphi\!\left(\frac{N}{d}\right).
\]In particular, let $d_r$ be the smallest positive integer such that $\mu_{d_r}=r$. Then $d_r\mid N$ and $\operatorname{mult}(r)\ge \varphi\!\left(\frac{N}{d_r}\right)$.
\end{proposition}

\begin{proof}
 By Proposition \ref{prop:gcd reduction}, we have $k\in \mathcal{A}_r(N,S)\iff \gcd(k,N)\in \mathcal{A}_r(N,S)$. 

Hence
\[
\mathcal{A}_r(N,S)=\coprod_{\substack{d\mid N\\ \mu_d=r}} \left\{1\le k\le N-1:\  \gcd(k,N)=d\right\}.
\]Therefore,
\[
\operatorname{mult}(r)=|\mathcal{A}_r(N,S)|=\sum_{\substack{d\mid N\\ \mu_d=r}}\varphi\!\left(\frac{N}{d}\right).
\]
\end{proof}

\subsection{Multiplicity bounds for nonzero eigenvalues under small support}\label{sec:nonzero-multiplicity}

In this subsection, we prove Theorem \ref{thm:m(u)}. We first consider the case when $N$ is even.

\begin{theorem}\label{thm:mle6}
Let $N>3$ be even. Let $S=\{\pm s_1,\pm s_2\}\subset \Z/N\Z$ be a symmetric generating set and assume $S$ contains a unit.
Then for every nonzero eigenvalue $\mu$, $\operatorname{mult}(\mu)\le 6$, and this bound is optimal.
\end{theorem}

\begin{proof}
The bound $6$ is optimal. Take $N=30$ and $S=\{\pm 1,\pm 3\}$. Then $\mu_k
=2\cos\frac{k\pi}{15}+2\cos\frac{k\pi}{5}$. And we have $\mu_k=1$ for $k\in\{3,9,10,20,21,27\}$. Hence $\operatorname{mult}(1)=6$.

Now we prove the bound. Choose a unit $u\in S$. Multiplication by $u^{-1}$ is an automorphism of $\mathbb Z/N\mathbb Z$,
so we may assume $S=\{\pm1,\pm a\}$. Hence
\[
\mu_k=2\cos\frac{2k\pi}{N}+2\cos\frac{2ak\pi}{N}
\qquad (k=0,1,2,\dots,N-1).
\]
Fix a nonzero eigenvalue $\mu$ and set $c := \mu/2 \neq 0$. For each $0 \le k \le N/2$, we define
 $b(k) := \min(\overline{ak}, N - \overline{ak})$, where $0\le \overline{ak}:=ak\pmod{N}\le N-1$. It follows that $\frac{2k\pi}{N}, \frac{2\pi\cdot b(k)}{N} \in [0, \pi] \cap \mathbb{Q}\pi$ for all $k \in \{0, 1, \dots, N/2\}$. Define the set
\[
K := \Bigl\{ k \in \{0, 1, \dots, N/2\} : \mu_k=\cos\frac{2k\pi}{N} + \cos\frac{2\pi\cdot b(k)}{N} = c \Bigr\}.
\]
Since $\mu_k=\mu_{N-k}$, we have $\operatorname{mult}(\mu)\le 2|K|$. Thus it is enough to prove $|K|\le3$. Suppose for contradiction that $|K|\ge 4$.

For $k\in K$, we have $\Theta(k):=\{\frac{2k\pi}{N},\frac{2\pi\cdot b(k)}{N}\}\in {\mathcal R}(c)$, where $\mathcal{R}(c)$ is defined in Proposition \ref{prop:three-reps}. We claim that if $\Theta(k)=\Theta(\ell)$ for distinct $k,\ell\in K$, then the two angles in $\Theta(k)$
have the same reduced denominator. Indeed, $\Theta(k)=\Theta(\ell)$ and $k\neq \ell$ imply $b(k)=\ell$ and $b(\ell)=k$, so $\gcd(\ell,N)=\gcd(ak,N)\ge \gcd(k,N)$ and 
$\gcd(k,N)=\gcd(a\ell,N)\ge \gcd(\ell,N)$. Thus $\gcd(k,N)=\gcd(\ell,N)$. Therefore, any pair in $\mathcal R(c)$ whose two angles have different reduced
denominators can arise from at most one $k\in K$, while a pair with equal reduced denominators can
arise from at most two since $|\Theta(k)|\le 2$.

  Fixed $c\neq0$, at most one pair in
$\mathcal R(c)$ has equal reduced denominators by Proposition~\ref{prop:three-reps}. Theorefore, at most one pair in $\mathcal{R}(c)$ has a two-element fiber under $k \mapsto \Theta(k)$, while all other fibers have size at most one. Note that $|\mathcal{R}(c)| \le 3$ by Proposition~\ref{prop:three-reps}. Consequently, if $|K| \ge 4$, then necessarily $|\mathcal{R}(c)| = 3$, every pair in $\mathcal{R}(c)$ occurs as $\Theta(k)$ for some $k$, and the unique equal-denominator pair occurs exactly twice. By Proposition~\ref{prop:three-reps}, $c\in\{\pm\frac12,\ \pm\cos\frac{\pi}{5},\ \pm\cos\frac{2\pi}{5}\}$ and $\mathcal R(c)$ is one of the six explicit triples listed there. In each such
triple there is exactly one equal-denominator pair; the other two are of denominator-type
$(1,3)$ and $(2,3)$ when $c=\pm\frac12$, and of denominator-type $(2,5)$ and $(3,5)$
for the four remaining values. Note that pairs of type $(2,3)$, $(2,5)$, and $(3,5)$ are impossible. Indeed, if
\[
\Theta(k)=\Bigl\{\frac{u\pi}{m},\frac{v\pi}{n}\Bigr\}
\qquad\bigl((m,n)\in\{(2,3),(2,5),(3,5)\}\bigr),\qquad \gcd(u,m)=\gcd(v,n)=1,
\]
then either $k=\frac{uN}{2m}$, $b(k)=\frac{vN}{2n}$, or the reverse, and hence
\[
un\,a\equiv \pm vm \pmod{2mn}
\qquad\text{or}\qquad
vm\,a\equiv \pm un \pmod{2mn}.
\]
For $(m,n)=(2,3)$, one congruence has left-hand side divisible by $3$ and right-hand side not,
while the other has left-hand side even and right-hand side odd; for $(2,5)$ the same argument
uses $2$ and $5$, and for $(3,5)$ it uses $3$ and $5$. Thus none of these types can occur. Thus, not every pair in $\mathcal{R}(c)$ has the form $\Theta(k)$, a contradiction. Hence $|K|\le 3$ and $\operatorname{mult}(\mu)\le 2|K|\le 6$.
\end{proof}

\begin{example}\label{ex:N even S=4}
 If we remove the condition that $S$ contains an invertible element, the bound will be larger. For example, when $N=60$ and $S=\{\pm 3,\pm 4\}$ (which still
generates $\mathbb Z/60\mathbb Z$), the nonzero eigenvalue $\mu=-1$ occurs with
multiplicity $8$.
\end{example}

\begin{lemma}\label{lem:supp-rigidity-to-mult}
Let $S \subseteq \mathbb{Z}/N\mathbb{Z}$ be a symmetric generating set that contains a unit, and fix $\mu \neq 0$. Assume further that $S$ satisfies\[\mu_{k_1} = \mu_{k_2} \neq 0 \quad \Longrightarrow \quad k_1S = k_2S .
\]
Then, the multiplicity satisfies $\operatorname{mult}(\mu) \le |S|$.

\end{lemma}

\begin{proof}
Fix an integer $k$ such that $0\le k<N$ and $\mu_k=\mu$. By assumption, we have $\operatorname{mult}(\mu)=|\{{0\le l<N }|{\text{ }lS=kS\text{ as multisets}}\}|$. We may assume that $j_1\in S$ is a unit. Let $0\le l<N$ such that $lS=kS$. Then we have $lj_1\in kS$. Hence $lj_1\equiv y\mod N$ for some $y\in kS$. Then $l\equiv (j_1)^{-1}y\mod N$ for some $y\in kS$. Since $0\le l<N$ and $|kS|\le |S|$, there are at most $|S|$ choices of $y$ and hence $|S|$ choices of $l$. Hence $\operatorname{mult}(\mu)\le |S|$.    
\end{proof}
\begin{theorem}\label{thm:mleS}
Let $N$ be an odd number whose smallest prime factor is $p$, and let $S\subset \Z/N\Z$ be a symmetric generating set with $|S|\le 2p$.
Assume $S$ contains a unit. Then for every nonzero eigenvalue $\mu$, $\operatorname{mult}(\mu)\le |S|$. The bound is optimal.
\end{theorem}

\begin{proof}
   By Theorem \ref{thm:rig}, the set $S$ satisfies the conditions in  Lemma \ref{lem:supp-rigidity-to-mult}. Then the result follows from Lemma \ref{lem:supp-rigidity-to-mult}.

   The bound is optimal. Let $N=p^aq^b$ where $p<q$ are two distinct odd primes and $a,b$ are two positive integers. We assume that $a\ge 2$ or $q\equiv 1 \mod p$, then the congruence $x^p\equiv 1 \mod N$ has non-trivial solutions. Let $m$ be a non-trivial solution and $S=\{\pm 1,\pm m,\pm m^2,\dots,\pm m^{p-1}\}$. Then we have $S=mS=\cdots m^{p-1} S$. Hence $\operatorname{mult}(\mu_1)\ge 2p$. This implies that the bound  is optimal.
\end{proof}

The condition that $S$ contains a unit cannot be removed
 \begin{example}\label{example:N/p,N/q}
The condition that $S$ contains a unit cannot be removed. Let $p < q$ be two odd prime numbers and $N = pq$. Let $S = \{p, 2p, \dots, (q-1)p, q, 2q, \dots, (p-1)q\}$. Since $\gcd(p,q) = 1$, the set $S$ is a symmetric generating set, and $S$ does not contain any unit. For any $1 \le k \le N-1$ with $\gcd(k, pq) = 1$, we have $\mu_k = \sum\limits_{j=1}^{q-1} \zeta_q^{kj} + \sum\limits_{j=1}^{p-1} \zeta_p^{kj} = -2$. Thus, the multiplicity $\operatorname{mult}(-2) \ge (p-1)(q-1)$, which is not bounded by $|S| = p+q-2$.

\end{example}

\begin{example}
    The zero eigenvalue can have a large multiplicity
even when the generating set  $S$ is small. Let $N$ be an interger whose smallest prime factor is $p$ and take $S=A\left(1;\frac{N}{p}\right)$.  Then $|S|=2p$.  But for every $k$ with $p\nmid k$, we have  $\mu_k=\ev_N(kS)=0$. Thus the zero eigenvalue has multiplicity $\operatorname{mult}(0)$ at least $N-N/p$, which is not bounded by $|S|$.
\end{example}

 \begin{cor}\label{cor:mu1-zero-or-irrational}
Let $N$ be an odd composite number, let $p$ be the smallest prime factor of $N$, and let
$S=\{\pm s_1,\pm s_2,\dots,\pm s_t\}\subset \mathbb Z/N\mathbb Z$ be a symmetric generating set such that $S$ contains a unit
and $|S|\le 2p$. Then $\mu_1=\sum\limits_{j=1}^t 2\cos(2\pi s_j/N)$ is either $0$ or irrational. Moreover, $\mu_1=0$ if and only if  $|S|=2p$ and $S=A(u;N/p)$ for some integer $u$.
\end{cor}

\begin{proof}
Assume  that $\mu_1\in \mathbb Q$ and $\mu_1\neq 0$. By Proposition \ref{prop:mr exact} and Theorem~\ref{thm:mleS}, we have $\varphi(N)\le\operatorname{mult}(\mu_1)\le 2p$. If $N$ has at least two distinct prime
factors $p<p_2$, then
\[
\varphi(N)\ge (p-1)(p_2-1)\ge (p-1)(p+1)=p^2-1>2p,
\]
a contradiction. If $N\ne 9$ and $N=p^r$ with $r\ge 2$, then $\varphi(N)=p^{\,r-1}(p-1)>2p$, a contradiction. The remaining case $N=9$ can be  excluded  by direct computation.  Hence $\mu_1$
is either $0$ or irrational.

Now assume that $\mu_1=0$. Theorem \ref{thm:LamLeung}(3) and Theorem \ref{thm:zero-2p1} imply that $S=A(u;N/p_1)$ for some integer $u$.
\end{proof}

\begin{example}
  In the above corollary, the condition that $N$ is a composite number is necessary. Let $N=\ell$ be an odd prime number. Let $S=(\Z/\ell\Z)\setminus\{0\}$. Then $S$ is symmetric, contains a unit, and satisfies $|S|=\ell-1<2\ell$. However, $\mu_1=-1\in\Q$.
\end{example}

\subsection{\texorpdfstring{$N$ square-free with $S\le(\mathbb{Z}/N\mathbb{Z})^\times$}{N square-free with S <= (Z/NZ)*}}\label{sec:squarefree subgroup}

In this subsection, we use Galois theory to compute the eigenvalue multiplicities of the graph $(\mathbb{Z}/N\mathbb{Z}, S)$, where $N$ is square-free and $S \le (\mathbb{Z}/N\mathbb{Z})^\times$ is a subgroup.

We recall some facts from Galois theory.

\begin{definition}\label{def:normal basis}
	Let $L/K$ be a finite Galois extension with Galois group $G=\Gal(L/K)$. A basis for $L$ as a $K$-vector space is called a normal basis if it is of the form $\{\sigma(\alpha):\sigma\in G\}$ for some $\alpha\in L$.
\end{definition}

\begin{example}\label{ex:normal basis}
	Let $p$ be a prime, the basis $\{1,\zeta_p,\dots,\zeta_p^{p-2}\}$ of $\mathbb{Q}(\zeta_p)/\mathbb{Q}$ is not a normal basis since $1$ is not conjugate to the rest of the basis, but $\{\zeta_p,\zeta_p^2,\dots,\zeta_p^{p-1}\}$ is a normal basis.
\end{example}

In general, we have the following fact about normal bases of cyclotomic fields.

\begin{proposition}\label{prop:normal basis cyclo}
	The primitive $n$-th roots of unity form a normal basis for $\mathbb{Q}(\zeta_n)/\mathbb{Q}$ if and only if $n$ is square-free.
\end{proposition}

\begin{proof}
For the sufficiency, see \cite[Lemma (2.11)]{lenstra1978vanishing}. For the necessity, suppose that $n$ is not square-free and consider the sum of all primitive $n$-th roots of unity. By \cite[(16.6.4)]{hardy2008}, we have
\begin{equation}\label{eq:ram}
    \mu(n) = \sum_{\substack{1 \le k \le n \\ \gcd(k,n)=1}} \zeta_{n}^{k},
\end{equation}
where $\mu(n)$ is the Möbius function. Since $n$ is not square-free, we have $\mu(n) = 0$. Hence, the set of primitive $n$-th roots of unity is linearly dependent, and thus cannot form a basis for $\mathbb{Q}(\zeta_n)/\mathbb{Q}$. 
\end{proof}

Using a normal basis, we can find a primitive element for every intermediate extension in a finite Galois extension.

\begin{proposition}\label{prop:normal basis}
	Let $L/K$ be a finite Galois extension with Galois group $G$ and let $\{\sigma(\alpha)|\sigma\in G\}$ be a normal basis. For a subgroup $H\subset G$, we have $L^{H}=K(\alpha_H)$ where $\alpha_H=\sum\limits_{\tau\in H}\tau(\alpha)$.
\end{proposition}

\begin{proof}
	Since $\alpha_H\in L^H$, we have $K(\alpha_H)\subset L^H$. By the Galois
	correspondence, it suffices to show that
  $\Gal(L/K(\alpha_H))=H$. If $\sigma\in H$, then $\sigma(\alpha_H)=\sum\limits_{\tau\in H}\sigma\tau(\alpha)=\sum\limits_{\tau\in H}\tau(\alpha)=\alpha_H$, so $H\subset \Gal(L/K(\alpha_H))$.
	
	Conversely, let $\sigma\in \Gal(L/K(\alpha_H))$. Then $\sigma(\alpha_H)=\alpha_H$, so $\sum\limits_{\tau\in H}(\sigma\tau)(\alpha)=\sum\limits_{\tau\in H}\tau(\alpha)$. Since $\{\gamma(\alpha):\  \gamma\in G\}$ is a basis of $L$ over $K$, these two
	sums can be equal only if they involve the same basis elements. Hence $\sigma H=H$. Because $1\in H$, this implies $\sigma=\sigma\cdot 1\in \sigma H=H$. Therefore
	$\Gal(L/K(\alpha_H))\subset H$. So $\Gal(L/K(\alpha_H))=H$, and the Galois correspondence yields $K(\alpha_H)=L^H$.	
\end{proof}

A direct consequence of Propositions \ref{prop:normal basis cyclo} and \ref{prop:normal basis} is the following result:

\begin{cor}\label{cor:normal basis}
	Let $N$ be a positive square-free integer. Let $H\subset \Gal(\mathbb{Q}(\zeta_N)/\mathbb{Q})=(\mathbb{Z}/N\mathbb{Z})^\times$ be a subgroup. Then we have $\mathbb{Q}(\zeta_N)^{H}=\mathbb{Q}(\eta)$ where $\eta=\sum\limits_{s\in H}\zeta_N^s$.
\end{cor}

From now on, let $N$ be a square-free integer and let $U_N:=(\mathbb Z/N\mathbb Z)^\times$. Let $S\le U_N$ be a subgroup. The eigenvalues of the Cayley graph $(\mathbb Z/N\mathbb Z,S)$ are $\mu_k=\sum\limits_{s\in S}\zeta_N^{ks}$ where $ k=0,1,\dots, N-1$. Here we do not assume that $S$ is symmetric; hence, the Cayley graph $(G,S)$ may be directed, and the eigenvalues $\mu_k$ need not be sums of cosines.

For each divisor $m\mid N$, let $\pi_m:U_N\to U_m:=(\mathbb Z/m\mathbb Z)^\times$ be reduction modulo $m$, and let $S_m:=\pi_m(S)\le U_m$. We use the convention $U_1=S_1=\{1\}$ and $\zeta_1=1$.

\begin{definition}\label{def:gaussian-period}
	Let $m$ be a positive integer and let $H\le U_m$. For $c\in U_m$, the sum
	\[
	\eta_{m,H}(c):=\sum_{h\in H}\zeta_m^{ch}
	\]
	is called the Gaussian period of modulus $m$ attached to the coset $cH$.
\end{definition}

For $c\in U_m$, define
\[
\eta_{m,c}:=\eta_{m,S_m}(c)=\sum_{u\in cS_m}\zeta_m^u,
\qquad
\Lambda_{m,c}:=\frac{|S|}{|S_m|}\eta_{m,c}.
\]
These quantities depend only on the coset $cS_m\in U_m/S_m$. For each divisor $m\mid N$,
we decompose the indices $k\in \Z/N\Z$ into layers
\[
\Omega_m:=\left\{k\in \mathbb Z/N\mathbb Z:\frac{N}{\gcd(k,N)}=m\right\}
\]
and $\Omega_1=\{0\}$. For an eigenvalue $\lambda$ and each divisor $m\mid N$, we define the $m$-layer multiplicity of $\lambda$ by $\operatorname{mult}_m(\lambda):=
\left|\left\{k\in \Omega_m:\mu_k=\lambda\right\}\right|$
. The following theorem shows that the eigenvalues of the Cayley graph are exactly scaled Gaussian periods $\Lambda_{m,c}$.

\begin{theorem}\label{thm:squarefree-subgroup-spectrum}
	Let $N$ be square-free and let $S\le U_N$ be a subgroup.
	
	\begin{enumerate}[label=(\arabic*)]
		
		\item If $k\in \Omega_m$, then $k=(N/m)\ell$ for a unique
	$\ell\in U_m$ and $\mu_k=\Lambda_{m,\ell}$.
		
		\item For each fixed divisor $m$, the values $\Lambda_{m,c}$ are pairwise distinct as $cS_m$ ranges over $U_m/S_m$. Consequently, the $m$-th layer contributes exactly $[U_m : S_m] = \frac{\varphi(m)}{|S_m|}$ distinct eigenvalues, each having layer multiplicity $|S_m|$ and degree $[\mathbb{Q}(\Lambda_{m,c}) : \mathbb{Q}] = \frac{\varphi(m)}{|S_m|}$ over $\mathbb{Q}$.

	\end{enumerate}
\end{theorem}

\begin{proof}
	(1) If $k\in \Omega_m$, write $k=(N/m)\ell$ with $\ell\in U_m$. Then $\mu_k
	=
	\sum\limits_{s\in S}\zeta_N^{ks}
	=
	\sum\limits_{s\in S}\zeta_m^{\ell\,\pi_m(s)}$. Since $\pi_m|_S:S\to S_m$ is surjective and each element of $S_m$ has exactly
	$|S|/|S_m|$ preimages in $S$, we get
	\[
	\mu_k
	=
	\frac{|S|}{|S_m|}\sum_{u\in \ell S_m}\zeta_m^u
	=
	\Lambda_{m,\ell}.
	\]
	(2) For a fixed $m$, Corollary~\ref{cor:normal basis} implies
	\[
	\mathbb Q(\zeta_m)^{S_m}=\mathbb Q\!\left(\sum_{u\in S_m}\zeta_m^u\right).
	\]
	The Galois conjugates of $\eta_{m,1}$ are precisely the sums $\eta_{m,c}$ as $cS_m$ runs through $U_m/S_m$. Hence the $\eta_{m,c}$, and therefore the $\Lambda_{m,c}$, are pairwise distinct. Their number is $[U_m:S_m]$, and each has degree $[U_m:S_m]$. Since $|\Omega_m|=\varphi(m)$, each of these eigenvalues occurs exactly $|S_m|$ times in the $m$-th layer.
\end{proof}

\begin{definition}\label{def:period poly}
   For each divisor $m\mid N$, we define the $m$-th layer eigenvalue set $E_m$ and the polynomial $Q_m(X)$ by 
\[
E_m:=\{\Lambda_{m,c}:cS_m\in U_m/S_m\},\qquad Q_m(X):=\prod_{\alpha\in E_m}(X-\alpha).
\]
\end{definition}
If we write $a_m=|S|/|S_m|$ and $d_m=[U_m:S_m]$, then 
\[
Q_m(X)=a_m^{d_m}P_m\!\left(\frac{X}{a_m}\right),\qquad P_m(X):=\prod_{cS_m\in U_m/S_m}(X-\eta_{m,c}).
\]
Thus $Q_m$ is the rationally scaled version of the usual Gaussian-period polynomial $P_m$.

Moreover, $Q_m(X)$ is the minimal polynomial over $\mathbb Q$ of every element of $E_m$. Indeed, the Galois group $\operatorname{Gal}(\mathbb Q(\zeta_m)/\mathbb Q)$ is $U_m$, and $a\in U_m$ sends $\Lambda_{m,c}$ to $\Lambda_{m,ac}$. Thus the roots of $Q_m$ form one Galois orbit. By Theorem~\ref{thm:squarefree-subgroup-spectrum}, this orbit has size $[U_m:S_m]$, which is also the degree of $Q_m$. Hence $Q_m\in\mathbb Q[X]$ is irreducible. 

By Theorem~\ref{thm:squarefree-subgroup-spectrum}, we have
\[
\operatorname{mult}(\lambda)=
\sum_{m\mid N,\,\lambda\in E_m}|S_m|.
\]
Therefore, the total multiplicity problem reduces to determining when the layer sets $E_m$ and $E_n$ intersect for $m\ne n$. The next theorem gives a criterion for this.

\begin{theorem}\label{thm:layer-criterion}
	Let $N$ be square-free and let $m,n>1$ be divisors of $N$. Put $g=\gcd(m,n),r=\frac{m}{g}$ and $ t=\frac{n}{g}$. Let $\rho_{m,g}:U_m\to U_g$ and $\rho_{n,g}:U_n\to U_g$ be the natural reduction maps. Then the following conditions are equivalent:
	\begin{enumerate}[label=(\arabic*)]
		\item $E_m\cap E_n\ne\varnothing$; that is, the $m$-th and $n$-th layers share a common eigenvalue.
		\item $E_m=E_n$; that is, the two layers have the same set of eigenvalues.
		
		\item $\ker(\rho_{m,g})\subseteq S_m$, $\ker(\rho_{n,g})\subseteq S_n$ and $\frac{\mu(r)}{\varphi(r)}=\frac{\mu(t)}{\varphi(t)}$.
	\end{enumerate}
\end{theorem}

\begin{proof}
Firstly, the irreducibility of $Q_m$ and $Q_n$ establishes the equivalence among conditions (1) and (2). Indeed, if $E_m$ and $E_n$ share a value $\lambda$, then $\lambda$ is a root of both $Q_m$ and $Q_n$. Since both polynomials are monic and irreducible over $\mathbb Q$, both are the minimal polynomial of $\lambda$ over $\mathbb Q$, and hence $Q_m=Q_n$. Therefore $E_m=E_n$.
	
	It remains to prove the equivalence with (3). Suppose that $\lambda\in E_m\cap E_n\ne\varnothing$.  Since $N$ is square-free, we have $\Q(\zeta_m)\cap \Q(\zeta_n)=\Q(\zeta_g)$. Thus $\lambda\in \Q(\zeta_g)$. By Corollary \ref{cor:normal basis}, every element of $E_m$(resp. $E_n$) generates the fixed field $\Q(\zeta_m)^{S_m}$(resp. $\Q(\zeta_n)^{S_n}$) over $\mathbb Q$. Therefore
	\[
	\Q(\zeta_m)^{S_m}=\Q(\lambda)\subseteq \Q(\zeta_g)=\Q(\zeta_m)^{\ker(\rho_{m,g})},
	\qquad
	\Q(\zeta_n)^{S_n}=\Q(\lambda)\subseteq \Q(\zeta_g)=\Q(\zeta_n)^{\ker(\rho_{n,g})}.
	\]
	By the Galois correspondence, we have
	\[
	\ker(\rho_{m,g})\subseteq S_m,\qquad \ker(\rho_{n,g})\subseteq S_n.
	\]
    Now assume these two kernel inclusions. Since $\ker(\rho_{m,g})\subseteq S_m$ and $S_g=\rho_{m,g}(S_m)$, we have $S_m=\rho_{m,g}^{-1}(S_g)$. Since $m=gr$ with $\gcd(g,r)=1$, the Chinese remainder theorem gives $U_m\cong U_g\times U_r$. Under this identification, $S_m=\rho_{m,g}^{-1}(S_g)=S_g\times U_r$ and $|S_m|=|S_g|\varphi(r)$. Choose integers $b_g\in U_g,\,b_r\in U_r$ such that
		\[
		rb_g\equiv1\pmod g,
		\qquad
		gb_r\equiv1\pmod r.
		\]
		If $x\in\mathbb Z/m\mathbb Z$ corresponds to $(x_g,x_r)\in\mathbb Z/g\mathbb Z\times\mathbb Z/r\mathbb Z$, then $\zeta_m^x=\zeta_g^{b_gx_g}\zeta_r^{b_rx_r}$.  Let \(c\in U_m\) correspond to \((c_g,c_r)\in U_g\times U_r\).  Then $cS_m
        =(c_g,c_r)(S_g\times U_r)
        =c_gS_g\times U_r$. Thus, for $c\in U_m$, we have
		\[
		\eta_{m,c}
		=\sum_{x\in cS_m}\zeta_m^x=
		\sum_{v\in c_gS_g}\sum_{w\in U_r}\zeta_g^{b_gv}\zeta_r^{b_rw}
		=(\sum_{w\in U_r}\zeta_r^{b_rw})(\sum_{v\in c_gS_g}\zeta_g^{b_gv})=
		\mu(r)\eta_{g,b_gc_g},
		\]
		where the last equality follows from \eqref{eq:ram} and the definition of $\eta_{g,b_gc_g}$. Since $b_g\in U_g$, the map $c_gS_g\mapsto b_gc_gS_g$ is a bijection of $U_g/S_g$. Hence, as $cS_m$ runs over all cosets of $S_m$, $b_gc_gS_g$ runs through all cosets of $S_g$. Therefore
		\[
		\{\eta_{m,c}:cS_m\in U_m/S_m\}=\{\mu(r)\eta_{g,a}:aS_g\in U_g/S_g\}.
		\]
		Hence
\begin{equation}\label{eq:E_m}
    E_m=
		\frac{|S|\mu(r)}{|S_g|\varphi(r)}
		\{\eta_{g,a}:aS_g\in U_g/S_g\}.
\end{equation}
Similarly,
\begin{equation}\label{eq:E_n}
    E_n=
	\frac{|S|\mu(t)}{|S_g|\varphi(t)}
	\{\eta_{g,a}:aS_g\in U_g/S_g\}.
\end{equation}
If $E_m=E_n$, then taking the sum of all elements in each set gives $\frac{|S|\mu(r)}{|S_g|\varphi(r)}\mu(g)
	=
	\frac{|S|\mu(t)}{|S_g|\varphi(t)}\mu(g)$. Since $g$ is square-free, $\mu(g)\ne0$. Therefore $\frac{\mu(r)}{\varphi(r)}=\frac{\mu(t)}{\varphi(t)}$. Thus (1) imply (3).
	
	Conversely, if (3) holds, then \eqref{eq:E_m} and \eqref{eq:E_n} imply that $E_m=E_n$. 
\end{proof}

\begin{cor}\label{prop:small-omega}
	If $N$ is square-free and has at most three prime factors, then eigenvalues coming from different layers are distinct. In particular, if $\lambda\in E_m$, then $\operatorname{mult}(\lambda)=\operatorname{mult}_m(\lambda)=|S_m|$.
\end{cor}

\begin{proof}
	Suppose $E_m\cap E_n\neq\emptyset$ for some $m\ne n$, then Theorem~\ref{thm:layer-criterion} implies $\frac{\mu(r)}{\varphi(r)}=\frac{\mu(t)}{\varphi(t)}$. The integers $r=m/g$ and $t=n/g$ are coprime square-free divisors of $N$. Let $\omega(a)$ denote the number of distinct prime divisors of $a$, with $\omega(1)=0$. Since $N$ has at most three prime factors, we have $\omega(r)+\omega(t)\le $3.
	
	For any square-free $a$, the sign of $\mu(a)/\varphi(a)$ is $(-1)^{\omega(a)}$. Therefore equality implies $\omega(r)$ and $\omega(t)$ to have the same parity. Since $\omega(r)+\omega(t)\le 3$, we have 
    \[
     (\omega(r),\omega(t))\in \{(0,0),\ (1,1),\ (2,0),\ (0,2)\}.
    \]
    The first case gives $r=t=1$, hence $m=n$, a contradiction. In the second case, $r=p$ and $t=q$ are primes, and $\frac{\mu(r)}{\varphi(r)}=\frac{\mu(t)}{\varphi(t)}$ gives $-\frac1{p-1}=-\frac1{q-1}$, so $p=q$. Since $r$ and $t$ are coprime, this is impossible. In the remaining two cases, one of $r,t$ is $1$ and the other is a product of two distinct primes; then equality would imply $1/\varphi(a)=1$ for such a product $a$, which is impossible. Thus distinct layers cannot share eigenvalues.
\end{proof}

\begin{example}\label{ex:4-prime-rational-collision}
	The layer multiplicity need not equal the (total) multiplicity when $N$ has four prime factors. Let $N=3\cdot5\cdot7\cdot13=1365$,
	and $S=U_N=(\mathbb Z/N\mathbb Z)^\times$. Then $S_m=U_m$ for every divisor $m\mid N$, so each layer contributes exactly one eigenvalue
	\[
	\Lambda_m
	=\frac{\varphi(N)}{\varphi(m)}\sum_{u\in U_m}\zeta_m^u
	=\mu(m)\varphi(N/m),
	\]
	of layer multiplicity $\varphi(m)$. For $m=35$ and $n=39$, we have $g=1$, $r=35$, and $t=39$. The kernel conditions in Theorem~\ref{thm:layer-criterion} hold because $S_{35}=U_{35}$ and $S_{39}=U_{39}$, and $\frac{\mu(35)}{\varphi(35)}=\frac{1}{24}
	=\frac{\mu(39)}{\varphi(39)}$. Thus the eigenvalue $24$ has layer multiplicity $24$ in each of the two layers and total multiplicity $48$.
\end{example}

\vskip 6mm
	\noindent 
	{\bf Acknowledgements.} 
	The author would like to express sincere gratitude to his advisor, Prof. Yigeng Zhao, for his valuable discussions, helpful suggestions, and continuous support throughout this work.

\addcontentsline{toc}{section}{References}
%\bibliography{reference}
\newcommand{\etalchar}[1]{$^{#1}$}

\end{document}